\documentclass[12pt,a4paper]{amsart}

\usepackage{amssymb}
\usepackage{graphicx}
\usepackage[all]{xy}
\usepackage{geometry}
\usepackage{xcolor}
\usepackage{hyperref}

\newcommand{\gal}{{\rm Gal}}

\newtheorem{lemma}{Lemma}[section]

\newtheorem{theorem}[lemma]{Theorem}
\newtheorem{corollary}[lemma]{Corollary}
\newtheorem{proposition}[lemma]{Proposition}

\theoremstyle{remark}

\theoremstyle{definition}

\newtheorem{necs-cond}[lemma]{Necessary Condition}
\newtheorem{definition}[lemma]{Definition}
\newtheorem{observation}[lemma]{Observation}

\newcommand\oline[1] {{\overline{#1}}}

\def\divides{{\,|\,}}

\newcommand\lSylow[1]{{ #1^{(\ell)} }}
\newcommand\lsylow[1]{{ #1^{(\ell)} }}

\newcommand{\mQ}{\mathbb{Q}}\newcommand{\QQ}{\mathbb{Q}}
\newcommand\FF{{\mathbb{F}}}
\newcommand\ff[1]{{\mathbb{F}_{#1}}}
\newcommand\ab[1]{{{#1}^{\text{ab}}}}
\newcommand{\mZ}{\mathbb{Z}}\newcommand{\ZZ}{\mathbb{Z}}

\newcommand{\fp}{\mathcal{\frak{p}}}
\newcommand{\fq}{\mathcal{\frak{q}}}
\newcommand{\fQ}{\mathcal{\frak{Q}}}
\newcommand{\ra}{\rightarrow}
\long\def\thesis#1\endt{#1}

\DeclareMathOperator\indlim{\mathop{\underrightarrow{\lim}}\limits }

\DeclareMathOperator\Gal{Gal}

\DeclareMathOperator\Aut{Aut}

\DeclareMathOperator\chr{char}
\DeclareMathOperator\id{id}

\DeclareMathOperator\Un{un}
\DeclareMathOperator\Sc{sc}
\DeclareMathOperator\Ht{ht}

\DeclareMathOperator\Mod{\,mod\,}
\DeclareMathOperator\Hom{Hom}

\DeclareMathOperator\HLG{H}

\DeclareMathOperator\im{Im}

\begin{document}

\title
{The Sylow subgroups of the absolute Galois group $\gal(\mQ)$}%

\def\UMich{Deptartment of Mathematics, 530 Church St., University of Michigan, Ann Arbor 48109, USA. }
\def\TAU{School of Mathematical Sciences, Tel Aviv University, Ramat Aviv, Tel Aviv 69978, Israel}
\author{Lior Bary-Soroker}
\address{\TAU}
\email{barylior@post.tau.ac.il}
\author{Moshe Jarden}
\address{\TAU}
\email{jarden@post.tau.ac.il}
\author{ Danny Neftin}
\address{\UMich}
\email{neftin@umich.edu}%

\begin{abstract}
We describe the $\ell$-Sylow subgroups of  $\Gal(\mQ)$ for an odd prime $\ell$, by observing and studying their decomposition as $F\rtimes \mathbb{Z}_\ell$, where $F$ is a free pro-$\ell$ group, and $\mZ_\ell$ are the $\ell$-adic integers. 
We determine the finite $\mathbb{Z}_\ell$-quotients of $F$ and more generally show that every split embedding problem of $\mathbb{Z}_\ell$-groups for $F$ is solvable. Moreover, we analyze the $\mathbb{Z}_\ell$-action on generators of $F$.
\end{abstract}

\maketitle

\section{Introduction}
The absolute Galois group  $\Gal(K) = \Aut(\tilde{K}/K)$ of a field $K$ with algebraic closure $\tilde{K}$ is a central object in Galois theory. The most interesting case in number theory is $K=\QQ$, or more generally when $K$ is a number field. Despite an extensive study   (e.g.\ class field theory,  Galois cohomology, Galois representation, field arithmetic, etc.), a determination of the entire group $\Gal(K)$ is unlikely to be achieved in the foreseeable future.

When $K$ is an $\ell$-adic field much more is known. The maximal pro-$\ell$ quotient  of $\Gal(K)$ is completely understood by the consecutive works of Shafarevich, Demuskin, Serre, and Labute --- it admits a presentation with countably many generators subject to at most one relation, see \cite[\S5.6]{Ser}.
This led Serre to ask about a larger part of $\Gal(K)$, namely, its {\bf $\ell$-Sylow subgroups}.
Recall that profinite groups admit Sylow theory similar to that of finite groups \cite[\S 2.3]{RZ}. In particular: an $\ell$-Sylow subgroup of a profinite group $G$ is a maximal pro-$\ell$ subgroup of $G$; every two $\ell$-Sylow subgroups of $G$ are conjugate; and the maximal pro-$\ell$ quotient of $G$ is a quotient of an $\ell$-Sylow subgroup of $G$.


Answering Serre's question for an $\ell$-adic field $K$,  Labute \cite{Lab} gives a presentation of the $\ell$-Sylow subgroups of
$\Gal(K)$   with countably many generators subject to one relation. His strategy is to view an $\ell$-Sylow subgroup of  $\Gal(K)$ as an inverse limit of the maximal pro-$\ell$ quotients of $\Gal(K')$, where $K'$ ranges over finite extensions of $K$ of degree prime to $\ell$. 

When $K$ is a number field less is known about the maximal pro-$\ell$ quotient $Q$ of $\gal(K)$.
Presentations of $Q$ are known up to the second term of its descending $\ell$-central series and only under restrictive assumptions on $K$, see \cite[\S 11.4]{Koch}. Thus, Labute's strategy to studying the Sylow subgroups of $\gal(K)$ is not applicable  when $K$ is a number field.

We take a new approach to studying the $\ell$-Sylow subgroups of $\gal(K)$ 
whose starting point is the following observation.



For an $\ell$-Sylow subgroup $P$ of $\gal(K)$ denote by $\lsylow{K}$ its fixed field, so that $P=\gal(\lsylow{K})$.  Denote by $\mu_{\ell^\infty}$ the group of $\ell$-power roots of unity.


\begin{observation}\label{thm:obs}
Let $K$ be a number field and $\ell$ an odd prime.
Let  $Z$ be the Galois group $\Gal(\lSylow{K}(\mu_{\ell^\infty})/\lSylow{K})$ and $F = \Gal(\lSylow{K}(\mu_{\ell^\infty}))$. Then $Z$ is isomorphic to the group $\ZZ_\ell$ of $\ell$-adic integers, $F$ is free pro-$\ell$ group on countably many generators, and the $\ell$-Sylow subgroups of $\Gal(K)$ decompose as:
\begin{equation}\label{eq:cyclotomic_decomposition}
\Gal(\lSylow{K}) = F\rtimes Z.
\end{equation}
\end{observation}
Interpretations of splitting maps of \eqref{eq:cyclotomic_decomposition} and of generators of the tame part of $F$ are given in \S \ref{sec:cyc-dec}.

We call \eqref{eq:cyclotomic_decomposition} the \textbf{cyclotomic decomposition}.
To completely understand $\gal(\lsylow{K})$ it therefore remains to determine the action of the cyclic group $Z$ on $F$. We first determine the finite quotients of $F$ as a $Z$-group, and more generally study embedding problems for $F$ which respect the $Z$-action.

As in profinite group theory, in which embedding problems are used to determine profinite groups, we study the $Z$-group $F$ via $Z$-embedding problems.
A \textbf{finite $Z$-embedding problem} for $F$ is a pair of $Z$-epimorphisms $(\alpha\colon F\to \Gamma, \beta\colon G\to \Gamma)$, where $G,\Gamma$ are finite $Z$-groups. A \textbf{proper solution} of $(\alpha,\beta)$ is a lifting of $\beta$ to a $Z$-epimorphism $\gamma\colon F\to G$,  cf.\ \S\ref{sec:Zgrps}.

Analogously to the classical setting, solvability of $Z$-embedding problems is reduced to solvability of  Frattini $Z$-embedding problems and of split $Z$-embedding problems, see Proposition \ref{prop:frattini_split}. Here $(\alpha,\beta)$ is
\textbf{split} if $\beta$ has a section which is a  $Z$-homomorphism.

\begin{theorem}\label{thm:splitep}
Every finite split $Z$-embedding problem for $F$ is properly solvable.
In particular, every finite $\ell$-group $G$ equipped with a $Z$-action is a quotient of $F$ as a $Z$-group.
\end{theorem}

We note that in general Frattini $Z$-embedding problems for $F$ are not solvable. Nevertheless one can reduce such problems to a classical setting over global fields, see Proposition~\ref{Z-Frattini.prop}.

The proof of Theorem~\ref{thm:splitep} is based on the observation of Colliot-Th\'el\`ene that  fields with pro-$\ell$ absolute Galois group are ample \cite[Theorem~5.8.3]{JardenAP}, Pop's theorem on solvability of split embedding problems for function fields over an ample field   \cite[Theorem~5.9.2]{JardenAP}, and Hilbert's irreducibility theorem.

We then apply the resulting tools to make the first step towards determining the $Z$-action on $F$ by  describing the action on generators of $F$ up to elements in  $F^\ell[F,F]$, the first level in the lower $\ell$-central series of $F$.
That is,  we describe the structure of the Frattini quotient $\overline{F}=F/F^\ell[F,F]$ as a $Z$-module by determining its indecomposable direct $Z$-summands.

A $Z$-module $M$ is said to be a direct $Z$-summand of $\oline F$ of multiplicity $\kappa$, if $\oline F\cong M^\kappa\times M'$, where  $M^\kappa$ is the product of $\kappa$ copies of $M$, and $M'$ has no $Z$-summands isomorphic to $M$. Note that since $Z$ acts on the group ring $\FF_\ell[Z/\ell^n Z]$, it also acts on $\FF_\ell[[Z]]=\varprojlim \FF_\ell[Z/\ell^n Z]$.
\begin{theorem}\label{thm:direct-summands}
The indecomposable direct $Z$-summands of $\oline F$ are  $\ff{\ell}[[Z]]$ and \\ $\ff{\ell}[Z/\ell^nZ]$ for $n\in\mathbb{N}\cup\{0\}$. Each of these summands appears with multiplicity $\omega$. 
\end{theorem}
In analogy to the works of Demushkin, Serre and Labute, where the relations are determined up to elements in a low level of a filtration and then lifted to the entire group, Theorem \ref{thm:direct-summands} gives relations in a presentation of $\gal(\lsylow{K})$ up to elements in  the first level $F^\ell[F,F]$ of the lower $\ell$-central series of $F$.
Namely, letting $\sigma$ be a generator of $Z$,
each summand $\FF_{\ell}[Z/\ell^k Z]$ gives a subset of generators $x_1, \ldots, x_{\ell^k}$ of $F$ subject only to the relations $\sigma x_{i}\sigma^{-1} = x_ix_{i+1}$ for $i=1, \ldots, \ell^{k}-1$ and $\sigma x_{\ell^k}\sigma^{-1} =x_{\ell^k} y$, for some $y\in F^{\ell}[F,F]$. Similar relations are obtained for each $\FF_\ell[[Z]]$ summand, see Corollary~\ref{generators.cor}.

Our proof of Theorem \ref{thm:direct-summands} is based on the theory of Ulm invariants. In contrast to  the work of Min\'a\v c-Schulz-Swallow  \cite{MS}, \cite{MSS}, this approach also allows dealing with modules over an infinite group such as $Z$, see \S\ref{module.sec}.
Using this approach the proof reduces to determining the solvability of $Z$-embedding problems of elementary abelian $Z$-groups. We achieve the latter by establishing a local global principle using the Poitou-Tate duality theorem, and combining it with results from Iwasawa theory. 

If $K=\mQ$, we also deduce that $\oline F$ is not a direct product of indecomposable modules,  and hence not all generators of $\oline F$ arise from Theorem \ref{thm:direct-summands}. We show that obtaining a full account of the action on the remaining generators is equivalent to determining a certain 
Iwasawa module, cf. \S\ref{infinite-ulm.sec}.


We note that our methods are applicable and hence  also stated in greater generality over global fields and for $\ell=2$ as well. We are hopeful that the combination of our methods with Iwasawa theory and results of Efrat-Min\'a\v c \cite{EM} will shed light on the shape of relations up to higher levels of the lower $\ell$-central series of $F$, and advance us further towards a complete understanding of $\gal(\lsylow{\mQ})$.




\subsection*{Acknowledgments}

We thank Nguy\^e\~n Duy T\^an, Ido Efrat, Dan Haran, David Harbater, Jeffrey Lagarias, Jan Mina\v{c}, James Milne, Kartik Prasanna,
Jack Sonn, and Michael Zieve for helpful discussions, remarks and encouragement. The first author was supported by a Grant from the GIF, the German-Israeli Foundation for Scientific Research and Development.  This material is based upon work supported by the National Science Foundation under Award No. DMS-1303990.

\section{Embedding problems}
\subsection{$Z$-embedding problems}\label{sec:Zgrps}
Let $Z$ be a profinite group. A profinite $Z$-group is a profinite group $H$ together with a continuous $Z$-action. A $Z$-homomorphism $\phi\colon H_1\to H_2$ is a continuous homomorphism that commutes with the $Z$-action. We say that a subgroup $H_1$ of a profinite $Z$-group $H_2$ is a $Z$-subgroup, if the inclusion map $H_1\to H_2$ is a $Z$-homomorphism, that is, if $H_1$ is a closed subgroup that is closed under the action of $Z$.  A $Z$-embedding problem for a $Z$-group $H$, denoted by $(\alpha,\beta)$,  is a diagram
\begin{equation}\label{eq:EP}
\xymatrix{
	&H\ar[d]^{\alpha}\ar@{.>}[dl]_{\gamma}\\
G\ar[r]^{\beta}&\Gamma
}
\end{equation}
in which $G,\Gamma$ are profinite $Z$-groups and $\alpha, \beta$ are $Z$-epimorphisms. If $Z=1$, we recover the usual notion of embedding problems for profinite groups. A \textbf{solution} of the $Z$-embedding problem is a homomorphism $\gamma\colon H\to G$ that commutes the above diagram. A solution is called \textbf{proper} if it is surjective. A $Z$-embedding problem is called \textbf{split}  if $\beta$ has a section which is $Z$-morphism. We define the \textbf{$Z$-Frattini} subgroup $\Phi_Z(G)$ of a $Z$-profinite group $G$ to be the intersection of all maximal $Z$-subgroup. We call a $Z$-embedding problem, as above,  \textbf{Frattini} if $\ker \beta\leq \Phi_Z(G)$.
If $G$ is finite (and hence so is $\Gamma$) we say that the $Z$-embedding problem is \textbf{finite}.
In this work we will be interested in $Z=\mathbb{Z}_\ell$ or $Z=1$.

\begin{lemma}\label{lem:Zbasis}
If $U$ is an open subgroup of a profinite $Z$-group $H$, then $U_Z=\bigcap_{z\in Z} U^z$ is open in $H$.
\end{lemma}

\begin{proof}
Since the action map $p\colon H\times Z\to H$ is continuous, $p^{-1}(U)$ is open. Thus there exist open normal subgroups $H_0\leq H$ and $Z_0\leq Z$ such that $p^{-1}(U)$ is a finite union of cosets of $H_0\times Z_0$, say $p^{-1}(U) = \bigcup_{i=1}^n  H_0{h_i} \times  Z_0{z_i}$. Thus
\[
U_Z= \bigcap_{z\in Z} U^z = \bigcap_{z\in Z} \bigcup_{i=1}^n (H_0h_i)^{ Z_0z_iz } =
 \bigcap_{z\in Z} \bigcup_{i=1}^n (H_0h_i)^{Z_0z^{z_i^{-1}}z_i }=\bigcap_{x\in Z/Z_0} \bigcup_{i=1}^n (H_0h_i)^{Z_0x^{z_i^{-1}}z_i }.
\]
We conclude that $U_Z$ is open as a finite intersection of open sets.
\end{proof}

Most of the basic theory of embedding problems carries over to $Z$-embedding problems. The proofs are similar to the classical case $Z=1$. For the sake of completeness, we prove the properties we shall need.

\begin{lemma}\label{lem:solutionfrattini}
If $(\alpha\colon H \to \Gamma, \beta\colon G\to \Gamma)$ is a Frattini $Z$-embedding problem and if $\gamma\colon H\to G$ is a solution, then $\gamma$ is proper.
\end{lemma}

\begin{proof}
Let $U=\gamma(H)$. If $U\neq G$, then there is a maximal $Z$-subgroup $V$ of $G$ that contains $U$. So
\[
\Gamma = \alpha(H)= \beta(\gamma(H)) = \beta(U)\leq \beta(V).
\]
By the third isomorphism theorem this implies that $G=V\ker\beta$. Since $(\alpha,\beta)$ is Frattini, $\ker \beta\leq \Phi_Z(G)\leq V$. So $G = V \ker \beta\leq V\Phi_Z(G)\leq V\neq G$. This contradiction implies that $U=G$, as needed.
\end{proof}

The following lemma reduces the study of solvability of embedding problems to the study of Frattini and split embedding problems.
\begin{proposition}\label{prop:frattini_split}
Consider a  $Z$-embedding problem $\mathcal{E} = (\alpha\colon H\to \Gamma, \beta\colon G\to \Gamma)$ for a $Z$-profinite group $H$. Then there exists an open $Z$-subgroup $U$ of $G$ such that $\beta(U)=\Gamma$ and the following properties are satisfied:
\begin{enumerate}
\item \label{lem:frattini_split1} The $Z$-embedding problem $\mathcal{E}_U=(\alpha \colon H\to  \Gamma, \beta|_{U}\colon U\to \Gamma)$ is Frattini.
\item \label{lem:frattini_split2} A solution $\alpha'\colon H\to U$ of $\mathcal{E}_U$ induces a split $Z$-embedding problem $\mathcal{E}'=(\alpha'\colon H\to U, \beta' \colon \ker \beta\rtimes U\to U)$, where $U$ acts on $\ker \beta$ by conjugation in~$G$.
\item \label{lem:frattini_split3} A proper solution $\gamma'\colon H\to \ker\beta \rtimes U$ of $\mathcal{E}'$ induces a proper solution $\gamma\colon H\to~G$ of $\mathcal{E}$ by: $\gamma'(h) =(\sigma,u)$ implies $\gamma(h) = \sigma u$.
\end{enumerate}
\end{proposition}

\begin{proof}
A limit argument reduces the proof to finite $Z$-embedding problems.

Let $U$ be minimal among the open $Z$-subgroups of $G$ that map onto $\Gamma$. In particular $\beta(U)=\Gamma$. Since no proper $Z$-subgroup of $U$ maps onto $\Gamma$, we have that $\ker (\beta|_U)$ is contained in each of the maximal $Z$-subgroups of $U$, hence $\ker(\beta|_U)$ is contained in $\Phi_Z(U)$. This proves \eqref{lem:frattini_split1}.

If $\alpha'$ is a solution of $\mathcal{E}_U$, then it is proper by Lemma~\ref{lem:solutionfrattini}. To prove  \eqref{lem:frattini_split2}, it suffices to observe that $\ker\beta\rtimes U$ is a profinite $Z$-group with respect to the action $(\sigma,u)^z= (\sigma^z,u^z)$ and that the projection map $\beta'\colon \ker\beta\rtimes U\to U$ is a $Z$-map.

Let $\pi\colon \ker \beta\rtimes U \to G$ defined by $\pi(\sigma ,u)=\sigma u$. It is a $Z$-epimorphism that commutes in the diagram of  $Z$-maps
\[
\xymatrix{
&&H\ar[dd]^{\alpha}\ar[dl]^{\alpha'}\ar@{.>}[dll]_{\gamma'}\\
\ker\beta \rtimes U\ar[r]^-{\beta'}\ar[dr]^{\pi} &U\ar@{_(->}[d]\ar[dr]|{\beta|_U}\\
&G\ar[r]_{\beta}&\Gamma.
}
\]
Thus if $\gamma'$ is a proper solution of $\mathcal{E}'$, then $\gamma$ is a proper solution of $\mathcal{E}$, as needed for~\eqref{lem:frattini_split3}.
\end{proof}

\begin{lemma}\label{definition.lem}
Let $H_1$ be a $Z$-subgroup of a profinite $Z$-group $H$ and let  $\alpha_1\colon H_1\to \Gamma$ be a $Z$-epimorphism on a finite $Z$-group $\Gamma$.
Then  there exists an open $Z$-subgroup $H_2$ of $H$ that contains $H_1$ and an extension $\alpha_2\colon H_2\to \Gamma$ of $\alpha_1$.

In particular any finite $Z$-embedding problem for $H_1$ is the restriction of a corresponding $Z$-embedding problem for an open $Z$-subgroup of $H$ that contains $H_1$.
\end{lemma}

\begin{proof}
The subgroup $U_1=\ker\alpha_1$ is a normal open $Z$-subgroup of $H_1$. Then there exists an open normal subgroup $U_2$ of $H$ such that $U_2\cap H_1\leq U_1$. By Lemma~\ref{lem:Zbasis} we may replace $U_2$ by $\bigcap_{z\in Z} U_2^z$ to assume that $U_2$ is a $Z$-subgroup.

Let $H_2=U_2H_1$. Then $H_2$ is an open $Z$-subgroup of $H$ that contains $H_1$. Let $\alpha_2\colon H_2\to \Gamma$ be defined by $\alpha_2(u\sigma)=\alpha_1(\sigma)$ for all $u\in U_2$ and $\sigma \in H_1$.
Then $\alpha_2$ is well defined because it is trivial on $U_2\cap H_1\leq U_1$ and it is a $Z$-map because its kernel $U_2$ is an open normal $Z$-subgroup. By definition $\alpha_2|_{H_1}=\alpha_1$, hence the assertion.
\end{proof}

We shall need the following two basic lemmas concerning Sylow subgroups of profinite groups:
\begin{lemma}\label{restriction.lem}
Let $\ell$ be a prime number, $\Lambda$  an $\ell$-Sylow subgroup of $G$, and $\alpha\colon G\to H$ an epimorphism of profinite groups. Assume that  $H$ is pro-$\ell$. Then $\alpha(\Lambda) = H$.
\end{lemma}

\begin{proof}
The notation $[A:B]$ denotes the index of a subgroup $B$ of a profinite group as a supernatural number, cf.\ \cite[\S 22.8]{FJ}. By the isomorphism theorems for profinite groups one has
\[
 [H:\alpha(\Lambda)] = [G:\Lambda \ker \alpha] .
\]
Since $H$ is pro-$\ell$ the left hand side is a (supernatural) power of $\ell$. Since $\Lambda$ is an $\ell$-Sylow subgroup, the right hand side, which divides $[G:\Lambda]$, is prime to $\ell$. Hence $[H:\alpha(\Lambda)]=1$, as needed.
\end{proof}

\begin{lemma}\label{subgroup.lem}
Let $\ell$ be a prime number and $H$ a normal subgroup of a profinite group $G$. Assume $[G:H]$ is prime to $\ell$. Then $H$ contains all $\ell$-Sylow subgroups of~$G$.
\end{lemma}

\begin{proof}
Let $\Lambda$ be an $\ell$-Sylow subgroup of $H$. Then $[G:\Lambda]=[G:H][H:\Lambda]$ is prime to $\ell$ and so $\Lambda$ is an $\ell$-Sylow subgroup of $G$. Since $H$ is normal,  also $\Lambda^\sigma \leq H$ for all $\sigma\in G$. By the Sylow theorem every $\ell$-Sylow subgroup of $G$ is of the form $\Lambda^\sigma$, hence the assertion.
\end{proof}

%

Next we deal with restriction of embedding problems from Sylow subgroups.

\begin{lemma}\label{lem:res_solutions}
Let $\ell$ be a prime number, $H$ a profinite group, $\Lambda$ an $\ell$-Sylow subgroup, and $\mathcal{E}_\ell = (\alpha\colon \Lambda \to \Gamma, \beta\colon G\to \Gamma)$ a finite embedding problem with $G$ an $\ell$-group. Let $\mathcal{U}$ be the family of pairs $(U,\alpha_U)$ where $U$ is an open subgroup of $H$ containing~$\Lambda$ and $\alpha_U\colon U\to G$ extends $\alpha$.
\begin{enumerate}
\item \label{cond:res}
If there exists $(U,\alpha_U)\in \mathcal{U}$ such that $\mathcal{E}_U=(\alpha_U \colon U \to \Gamma, \beta \colon G\to \Gamma)$ has a solution $\gamma_U\colon U\to G$, then $\gamma=(\gamma_U)|_{\Lambda}$ is a solution of $\mathcal{E}$. Moreover if $\gamma_U$ is proper, then $\gamma$ is proper.
\item If $\ker \alpha$ is abelian and if $\mathcal{E}$ is solvable, then $\mathcal{E}_U$ is solvable.
\end{enumerate}
\end{lemma}

\begin{proof}
The first assertion of \eqref{cond:res}, that $\gamma$ is a solution of $\mathcal{E}$, is trivial. The second assertion of \eqref{cond:res} follows from Lemma~\ref{restriction.lem}.

Now we assume that $A=\ker \alpha$ is abelian and that $\mathcal{E}$ is solvable. Denote by $b$ the class in $H^2(\Gamma,A)$ that corresponds to the group extension
\[
\xymatrix@1{1\ar[r] & A\ar[r] &G\ar[r]^{\beta} & \Gamma \ar[r] & 1}
\]
and write $\alpha^*\colon H^2(\Gamma,A)\to H^2(\Lambda , A)$ for the inflation map. Then by Hoechsmann's theorem  \cite[Proposition 9.4.2]{NSW}, $\alpha^*(b)=0$.

Let $(U,\alpha_U)\in \mathcal{U}$ and let $i\colon \Lambda \to U$ be the inclusion map. Then $0=\alpha^*(b) = (\alpha_U\circ i)^*(b) = i^* \circ \alpha_U^*(b) $. Since $|A|$ is a power of $\ell$ and since $[U:\Lambda] \mid [H:\Lambda]$,  hence prime to $\ell$, it follows that $i^*$ is injective. So $\alpha_U^*(b) =0$ and consequently $\mathcal{E}_U$ is solvable by \cite[Proposition 9.4.2]{NSW}.
\end{proof}

We shall also need the following technical lemma:
\begin{lemma}\label{splitting.lem} Let $G$ be a profinite group, let $N$ and  $P$ be closed subgroups, and put $F=N\cap P$. Assume that $N\lhd G$, $G=NP$, and $P=F\rtimes Z$, for some $Z\leq P$. Then $G=N\rtimes Z$.
\end{lemma}
\begin{proof} Since
$ N\cap Z= N\cap P\cap Z=F\cap Z=1$ and $NZ= NFZ= NP = G$, we get the assertion.
\end{proof}

\section{The cyclotomic decomposition}
\label{sec:cyc-dec}

\subsection{Proof of Observation \ref{thm:obs}}
The following is a more general form of Observation \ref{thm:obs}.
\begin{observation}\label{semidirect2.thm} Let $K$ be a global field and $\ell\neq \chr(K)$ a prime. If $\ell=2$ and $K$ is a number field, assume further that $K\cap \mQ(\mu_{\ell^\infty})$ is (totally) imaginary.
Then $\Gal(\lSylow{K})\cong F \rtimes Z$, where $Z=\Gal(\lSylow{K}(\mu_{\ell^\infty})/\lSylow{K}) \cong \mathbb{Z}_\ell$ and $F=\Gal(\lSylow{K}(\mu_{\ell^\infty}))$ is  a free pro-$\ell$ group on countably many generators.
\end{observation}

\begin{proof}
Since $\mu_\ell\subseteq \lSylow{K}$ by Lemma \ref{subgroup.lem}, and since $\lsylow{K}\cap \mQ(\mu_{\ell^\infty})$ is (totally) imaginary if $K$ is a number field and $\ell=2$, one has $\Gal(\mQ(\mu_{\ell^\infty})/\lSylow{K}\cap \mQ(\mu_{\ell^\infty}))\cong \mathbb{Z}_\ell$. Thus,
 $Z\cong \mathbb{Z}_\ell$ as a nontrivial subgroup.
 The restriction map gives rise to a short exact sequence
\begin{equation}\label{short.equ}
\xymatrix@1{
1\ar[r]& F \ar[r]& \Gal(\lSylow{K}) \ar[r]^-\alpha&Z  \ar[r]& 1.}
\end{equation}
Since $\mathbb{Z}_\ell$ is projective in the category of pro-$\ell$ groups, (\ref{short.equ}) splits  and its splitting gives an isomorphism $\gal(\lSylow{K})\cong F\rtimes Z.$

Let $L=\lsylow{K}(\mu_{\ell^\infty})$. Since $L$ is totally imaginary and $[L_\fp:\mQ_p]$ is divisible by $\ell^\infty$ as a supernatural number for every rational prime $p$ and a prime $\fp$ of $L$ lying over $p$, the local Galois groups $\gal(L_\fp)$ has $\ell$-th cohomological dimension $1$ for every prime $\fp$ of $L$.
 The Albert-Brauer-Hasse-Noether theorem then shows that $F=\gal(L)$ has $\ell$-th cohomological dimension $1$, see \cite[Chp.~II \S3.3 Proposition 9]{Ser}. Thus,  $F$ is free pro-$\ell$ \cite[Chp.~I \S4 Corollary 2]{Ser}.
\end{proof}
\subsection{Existence of splitting maps} 
For $\ell=2$, if $K$ has a real prime, then the sequence
\[
\xymatrix@1{
1\ar[r]& \Gal(K^{(2)}(\mu_{2^\infty})) \ar[r]& \Gal(K^{(2)}) \ar[r]^-\alpha&\Gal(K^{(2)}(\mu_{2^\infty})/K^{(2)})  \ar[r]& 1
}
\]
does not split.
Otherwise, there is an embedding of  $$\Gal(K^{(2)}(\mu_{2^\infty})/K^{(2)}) \cong \Gal(K(\mu_{2^\infty})/K)\cong \ZZ_2\times \ZZ/2\ZZ$$ into $\Gal(K)$. But this is impossible since the normalizer of an involution $\tau$ in an absolute Galois group is exactly $\left<\tau\right>$, cf.\ \cite[Proposition~19.4.3(b)]{Ido}.

\begin{corollary}\label{2.cor}
Let $K$ be a number field equipped with a real prime. 
Then $$\gal(K^{(2)})\cong (F\rtimes \mathbb{Z}_2)\rtimes \mathbb{Z}/2,$$
where $F$ is a free pro-$2$ group on countably many generators.
\end{corollary}

\begin{proof} By Observation \ref{semidirect2.thm}, we have $\gal(K^{(2)}(\sqrt{-1}))\cong F \rtimes \mathbb{Z}_2$. Since $K$ has a real place, there is an embedding of $\oline{K}$ into $\mathbb{C}$ such that the complex conjugation $\tau$ fixes $K$. Thus, the restriction of $\tau$ to $\oline{K}$ is an involution which restricts to the nontrivial automorphism of $K^{(2)}(\sqrt{-1})/K^{(2)}$. This gives a splitting of the extension
\[
\xymatrix@1{
1\ar[r]&F\rtimes \mathbb{Z}_2 \ar[r]& \Gal(K^{(2)})\ar[r]& \Gal(K^{(2)}(\sqrt{-1})/K^{(2)})\ar[r]& 1,
}
 \]
 proving the desired result.
\end{proof}

If $K$ is totally imaginary but $K\cap \mQ(\mu_{2^\infty})$ is totally real, 
by Artin's theorem $\gal(K^{(2)})$ has no involutions and hence even the sequence   
$$ 1\ra F\rtimes \mathbb{Z}_2 \ra \Gal(K^{(2)})\ra \Gal(K^{(2)}(\sqrt{-1})/K^{(2)})\ra 1, $$
does not split. 
\subsection{Henselian splitting maps} 
We also note that a splitting $s:Z\ra \Gal(\lsylow{K})$ of \eqref{short.equ} can be chosen so that $s(Z)$ is generated by a lift of the Frobenius automorphism at any prime $\fp$ of $K$ such that $N(\fp)\not\equiv 1 \Mod \ell^{s+1}$, where $\ell^s$ is the number of $\ell$-power roots of unity in~$K(\mu_\ell)$. Indeed, letting $\frak{P}$ be a prime of $\tilde K$ dividing $\fp$, the condition on $N(\fp)$ forces $\frak{P}$ to be inert in $K(\mu_{\ell^{s+1}})/K(\mu_{\ell^s})$, and hence in $K(\mu_{\ell^\infty})/K(\mu_{\ell^s})$ and in $\lsylow{K}(\mu_{\ell^\infty})/\lsylow{K}$. Let $\sigma$ be any lift of the Frobenius of $\frak{P}$ in $\lsylow{K}(\mu_{\ell^\infty})/\lsylow{K}$ to $\tilde K$.  As $\frak{P}$ is inert in $\lsylow{K}(\mu_{\ell^\infty})/\lsylow{K}$, the restriction of $\sigma$ to $\lsylow{K}(\mu_{\ell^\infty})$ generates $Z$ and hence induces a splitting $s$ of~\eqref{short.equ}.

\subsection{Generators of the tame part of $F$} 

Let $L=\lsylow{K}(\mu_{\ell^\infty})$, $P$ (resp. $T$) the set of primes of $L$ (resp. primes of $L$ lying either over $\infty$ or $\ell$), and let $L_{T}$ be the maximal extension of $L$ unramified away from $T$. The number theoretical analogue of Riemann's existence theorem \cite[Corollary 10.5.2]{NSW} gives a canonical set of generators of $\gal(L_T)$. Namely, it shows that $\gal(L_T)$ decomposes as the free product of its local Galois groups:
$$ \Gal(L_T)\cong \mathop{\ast}_{\fp\in P\setminus T} \gal(L_\fp). $$
Here, $\ast$ denotes the free pro-$\ell$ product over the profinite index space associated to $P\setminus T$, see \cite[\S 10.1]{NSW}. Note that $\gal(L_\fp)$ is the inertia group, and hence is cyclic for every prime $\fp\not\in T$.  We also note that the abelinization of the remaining part $\gal(L_T/L)$  can be studied using Iwasawa theory. 

\section{The action via $Z$-embedding problems}
In this section we  study the action in the cyclotomic decomposition via $Z$-embedding problems. We consider the following more general setup. Let  $K$ be a Hilbertian  field and $\ell \neq \chr{K}$ a prime number. If $\ell=2$ and $\chr K=0$,  assume that $\sqrt{-1}\in K$. As before set $L=\lSylow{K}(\mu_{\ell^\infty})$, $Z=\gal(L/\lSylow{K})$, and $F=\gal(L)$.  Theorem \ref{thm:splitep} is then a special case of:
\begin{theorem}\label{thm:hilbertep}  Every finite split $Z$-embedding problem for $F$ is properly solvable.
\end{theorem}

To prove the theorem we first deal with split embedding problems for $\Gal(\lSylow{K})$:

\begin{proposition}\label{split.prop} Let $(\phi\colon\gal(\lSylow{K})\ra\Gamma, \pi\colon G\ra \Gamma)$ be a finite split embedding problem for $\gal(\lSylow{K})$ with $G$ an $\ell$-group. Then $(\phi,\pi)$ is properly solvable.
\end{proposition}

\begin{proof}
Let $N$ be the fixed field of $\ker\phi$, and so $N/\lSylow{K}$ is Galois and the map $\phi$ decomposes as $\phi = \phi'\circ r$, where $r\colon \Gal(\lSylow{K})\to \Gal(N/\lSylow{K})$ is the restriction map and $\phi'\colon \Gal(N/\lSylow{K})\to \Gamma$ is an isomorphism. We may replace $\Gamma$ by $\Gal(N/\lSylow{K})$ and the maps $\pi$, $\phi$ by $(\phi')^{-1}\circ \pi$ and $r$, respectively, to assume that $\Gamma=\Gal(N/\lSylow{K})$ and $\phi$ is the restriction map.

By \cite[Theorem~5.8.3]{JardenAP} $\lSylow{K}$ is ample. Hence by \cite[Theorem~5.9.2]{JardenAP} there exist a Galois extension $F/\lSylow{K}(x)$  such that $\Gal(F/\lSylow{K}(x)) \cong G$, $N$ is the algebraic closure of $K$ in $F$, and the restriction map $\Gal(F/\lSylow{K}(x)) \to \Gal(N(x)/\lsylow{K}(x))$ coincides with $\pi$ (after identifying $\Gal(F/\lSylow{K}(x)) = G$,  $\Gal(N(x)/\lSylow{K}(x))=\Gamma$).

Let $K_0$ be a finite subextension of $\lSylow{K}(x)/K$ to which the above descends to as follows: there exist $N_0/K_0$ Galois with Galois group $\Gamma$ such that $N=N_0\lSylow{K}$ and $F_0/K_0(x)$ Galois with group $G$ such that $F=F_0\lSylow{K}$, $N_0$ is the algebraic closure of $K_0$ in $F_0$, $G=\Gal(F_0/K_0(x))$ and the restriction map $\Gal(F_0/K(x))\to \Gal(N_0/K_0)$ coincides with $\pi$.

Note that $K_0$ is Hilbertian as a finite extension of $K$ \cite[Proposition 16.11.1]{FJ}.
Hence there exists $a\in K_0$ such that the prime $(x-a)$ of $K_0(x)$ is inert in $F_0$. Let $M$ be the residue field of $F_0$ at $x=a$. Then $M/K_0$ is Galois with Galois group $G$, $N_0\subseteq M$, and the restriction map $\Gal(M/K_0)\to \Gal(N_0/K_0)$ coincides with $\pi$. In other words, if $\phi_0\colon \Gal(K_0) \to \Gal(M/K_0)=G$ and $\psi\colon \Gal(K_0)\to \Gal(M/K_0)$ are the restriction maps, then $\psi$ is a proper solution of $(\phi_0,\pi)$. Then $\psi|_{\Gal(\lSylow{K})}$ is a solution of $(\phi,\pi)$ which is proper by Lemma~\ref{restriction.lem}.
\end{proof}
\begin{proof}[Proof of Theorem~\ref{thm:hilbertep}]
Let $(\phi\colon F\to G,\pi\colon G\to \Gamma)$ be a finite split $Z$-embedding problem with $G$ an $\ell$-group. Since $\Gal(\lSylow{K}) = F\rtimes Z$, we may extend $(\phi,\pi)$  to a split embedding problem
\[
(\phi'\colon F\rtimes Z\ra \Gamma\rtimes Z,\pi'\colon G\rtimes Z \ra \Gamma\rtimes Z )
\]
for $\Gal(\lSylow{K})$, where $\phi' (x,z) = (\phi(x),z)$ and $\pi'(g,z) = (\pi(g),z)$, for every $x\in F$, $z\in Z$, and $g\in G$.

Since $Z$ acts on the finite group $G$ continuously, the kernel of the action is an open subgroup of $Z$, so it contains  $\ell^r Z$, for some $r\geq 1$.
Composing with the natural projection $Z\to Z/\ell^r Z$ we obtain a finite embedding problem
\[
(\phi''\colon F \rtimes Z\ra \Gamma\rtimes (Z/\ell^rZ),\pi''\colon G\rtimes (Z/\ell^rZ)\ra \Gamma\rtimes (Z/\ell^rZ))
\]
for $\lSylow{K}$ and we have the commutative diagram of profinite groups
\begin{equation}
\xymatrix{
 & F \rtimes Z \ar_{\phi'}[d] \ar@/^20pt/[dd]^{\phi''} \\
 G\rtimes Z \ar^{\pi'}[r] \ar[d] &  \Gamma\rtimes Z \ar[d] \\
 G\rtimes (Z/\ell^rZ) \ar^{\pi''}[r] &  \Gamma\rtimes (Z/\ell^rZ).
 }
\end{equation}
By Proposition~\ref{split.prop}, there exists a proper solution $\psi''$ of $(\phi'', \pi'')$. 
Note that as $\ker\phi''=\ell^r\ker\phi'$, we have $\ker\psi''\ker\phi'=\ell^k\ker\phi'$ for some $k\geq r$. We claim that $k=r$ and hence \begin{equation}\label{span.equ}\ker\psi''\ker\phi'=\ker\phi''.\end{equation}
Indeed, if $k>r$, we have $\ker\psi''\ker\phi'\subseteq \ell^{r+1}Z\ker\phi'$ and hence $\pi''$ factors through the natural projection $\Gamma \rtimes Z/\ell^{r+1}Z \ra \Gamma\rtimes Z/\ell^rZ$. The latter does not split, contradicting the splitting of $\pi''$, and proving the claim.

Since $G\rtimes Z$ is the fiber product of $\Gamma\rtimes Z$ and $G\rtimes (Z/\ell^r Z)$ over $\Gamma\rtimes (Z/\ell^r Z)$, we obtain a solution $\psi'=\psi'' \times_{\phi''} \phi'$  of $(\phi',\pi')$.
We next show that $\psi'$ is proper. We have $\ker\psi'=\ker\psi''\cap \ker\phi'$. Hence \eqref{span.equ} gives:
$$\ker\phi'/\ker\psi' = \ker\phi'/(\ker\psi''\cap \ker\phi')\cong (\ker\phi'\ker\psi'')/\ker\psi''=\ker\phi''/\ker\psi''.$$
Thus, $[\ker\phi':\ker\psi']=[\ker\phi'':\ker\psi'']=[G\colon\Gamma]$, showing that $\psi'$ is surjective.  Since $\pi'$ and $\phi'$ are the identity maps on $Z$, $\psi'(F) = \im\psi' \cap G$. As $\psi'$ is proper, we get $\psi'(F)=G$. Thus,  the restriction of $\psi'$ to $F$ is a proper solution of the $Z$-embedding problem $(\phi,\pi)$.
\end{proof}

As oppose to split embedding problems, Frattini $Z$-embedding problems need not be solvable. We now descend these problems to cyclotomic extensions of number fields.

For a number field $K(\mu_\ell)\subseteq K'\subseteq \lSylow{K}$, Lemma \ref{splitting.lem} applied
with $N=\gal(K'(\mu_{\ell^\infty}))$ and $P=\gal(\lsylow{K})$ shows that the splitting $\gal(\lSylow{K})=\gal(L)\rtimes Z$ induces a splitting $\Gal(K')=\gal(K'(\mu_{\ell^\infty}))\rtimes Z$ such that the restriction $\Gal(L)\ra \gal(K'(\mu_{\ell^\infty}))$ is a $Z$-homomorphism.

\begin{proposition}\label{Z-Frattini.prop} Let $(\phi:\gal(L)\ra \Gamma,\pi)$  be a $Z$-embedding problem. Then there is a number field $K(\mu_\ell)\subseteq K'\subseteq \lSylow{K}$ and a $Z$-embedding problem $$(\phi':\gal(K'(\mu_{\ell^\infty}))\ra \Gamma,\pi)$$ whose restriction to $L$ is  $(\phi,\pi)$. If furthermore $\ker\pi$ is abelian, then for every such $K'$ and $\phi'$, $(\phi,\pi)$ is  solvable if and only if $(\phi',\pi)$ is  solvable.
In particular, if $\pi$ is $Z$-Frattini,  $(\phi,\pi)$ is properly solvable if and only if $(\phi',\pi)$ is properly solvable.
\end{proposition}
\begin{proof}
Let $N:=\gal(K(\mu_{\ell^\infty}))$ be a $Z$-group via the induced splitting $\gal(K(\mu_\ell))=N\rtimes Z.$ By Lemma \ref{definition.lem}, $\phi$ extends to $\phi':U\ra \Gamma$ for some open $Z$-subgroup $U\leq N$. Let $K'$ be the fixed field of $U\rtimes Z$. Since
 $UZ=U\gal(L)Z\supseteq \gal(\lSylow{K})$, we have $K'\subseteq \lSylow{K}$. Since $U\rtimes Z$ is open in $\gal(K(\mu_\ell))$, $K'$ is a number field.  Since $U\leq N$, $\mu_{\ell^\infty}$ is fixed by $U$ and $K'(\mu_{\ell^\infty})$ is the fixed field of $U$. Thus,   $\phi'$ is the desired $Z$-homomorphism. The equivalence for  solvability follows by Lemma \ref{lem:res_solutions}. Thus, the equivalence for proper solvability follows by Lemma \ref{lem:solutionfrattini}. 
\end{proof}




Explicit examples of  nonsolvable Frattini $Z$-embedding problems appear in the following section (Proposition \ref{finite-Ulm.prop}).


\section{Action on $F/F^\ell[F,F]$}
Let $\Gal(\lSylow{K})=F\rtimes Z$ be the cyclotomic decomposition for  a global field $K$ and a prime $\ell\neq \chr K$. If $K$ is a number field and $\ell=2$ we assume $\sqrt{-1}\in K$. Recall that $Z=\gal(L/\lSylow{K})\cong \mathbb{Z}_\ell$ and $F=\Gal(L)$ is a free pro-$\ell$ group, where $L=\lSylow{K}(\mu_{\ell^\infty})$.

To find the indecomposable direct $Z$-summands of $\oline F=F/F^\ell[F,F]$, we apply the theory of Ulm invariants for countably generated  $\ell$-torsion profinite $Z$-modules, basing on \cite[\S 11,12]{Kap} as described in the following section.

\subsection{$Z$-modules}\label{module.sec} 
Let $M$ be a countably generated profinite $Z$-module which is $\ell$-torsion, i.e. $\ell\cdot M=0$.
That is, $M$ is  a profinite  $\ff{\ell}[[Z]]$-module. 
The ring $\ff{\ell}[[Z]]$ is a discrete valuation ring whose maximal ideal is the augmentation  ideal $I=(\sigma-1)$, where $\sigma$ is a generator of $Z$.  Thus,  $I^nM,n\in\mathbb{N},$ is a fundamental system of open neighborhoods of $0\in M$.

As $M$ is profinite its (Pontryagin) dual $\hat M :=\Hom(M,\ff{\ell})$ is a discrete $\ff{\ell}[[Z]]$-module with the $Z$-action $(\tau f)(m)=f(\tau^{-1}m)$ for all $m\in M,\tau\in Z$, and $f\in \hat M$. Moreover, $\hat M$ is $\ff{\ell}[[Z]]$-torsion since
every homomorphism $f\in \hat M$ factors through  $M/I^nM$ for some $n\in\mathbb{N}$, so $I^nf=0$.


\begin{definition} For a discrete torsion $\ff{\ell}[[Z]]$-module $N$, let  $N^Z$ be the submodule of all element of $N$ fixed by $Z$, or equivalently  annihilated by $I$. Consider  the descending transfinite  sequence $I^nN$ defined by
$I^{n+1}N:=I(I^n)N$ for each ordinal $n$ 
and $I^nN=\bigcap_{k<n} I^kN$ for each limit ordinal $n$.
For every ordinal $n$, the {\bf Ulm invariant} $U_n(N)$ is the cardinality of $(I^nN)^Z/(I^{n+1}N)^Z$.
\end{definition}

The following proposition shows that the finite Ulm invariants $U_n(\hat M)$ already determine the finite $Z$-summands of $M$.

Since $\ff{\ell}[[Z]]$ is a complete discrete valuation ring, there is a unique cyclic $\ff{\ell}[[Z]]$-module $V_n:=\ff{\ell}[[Z]]/I^n$ of dimension $n$ over $\ff{\ell}$.
\begin{proposition}\label{bounded.prop}\label{ascending.rem}
Let $M$ be a profinite $\ff{\ell}[[Z]]$-module. Then
 $U_{n-1}(\hat M)$ is the multiplicity of $V_n$ as a direct $Z$-summand of $M$, for every $n\in\mathbb{N}$. Furthermore, for every $N\in\mathbb{N}$, $M=M_{\leq N}\times M_{>N}$, where $$M_{\leq N}\cong  \prod_{n\leq N}V_n^{U_{n-1}(\hat M)},$$ 
 $M_{>N}$  has  no direct $Z$-summands of dimension $\leq N$ over $\ff{\ell}$.
\end{proposition}
Proposition \ref{bounded.prop} follows from the theory of Ulm invariants and its proof is given in \S\ref{modules-proofs.sec}.

For $\eta\in  \hat M$ define  $\Ht(\eta)$ to be the maximal $n$ such that $\eta\in I^n\hat M$ if such an $n$ exists and $\infty$ otherwise\footnote{This height identifies with the height function defined in \cite{Kap}.}.
Thus, the Ulm invariants can be expressed using the height function as:
\begin{equation}\label{finite-Ulm.equ}  U_n(\hat M)=\left| \{\phi\in \hat M^Z\,|\, \Ht(\phi)\geq n \}/ \{\phi\in \hat M^Z\,|\, \Ht(\phi)> n \}\right|,\end{equation}
for $n\in\mathbb{N}\cup\{0\}$, and
\begin{equation}\label{infinite-Ulm.equ} I^\omega\hat M = \{\phi\in \hat M\,|\, \Ht(\phi)=\infty\}.\end{equation}

\subsection{The height via $Z$-embedding problems}

To compute the finite  Ulm invariants of $\hat F$ we first interpret the height in terms of $Z$-embedding problems.
Let $\pi_{n,m}\colon V_n\ra V_m$, and $\pi_m\colon \ff{\ell}[[Z]]\ra V_m$ denote the natural projections.

\begin{proposition}\label{height.prop} Let $M$ be a profinite $\ff{\ell}[[Z]]$-module, $k\in\mathbb{N}$, and $\eta\in \hat M$. Fix an $\ff{\ell}[[Z]]$-monomorphism $\tilde \eta\colon \hat V_m \ra \hat M$ whose image is $\ff{\ell}[[Z]]\eta$, where $m=\dim_{\ff{\ell}}\ff{\ell}[[Z]]\eta$. Let $\tilde\eta^*: M\ra V_m$ be its dual map.
Then
  $\eta\in I^k\hat M$ if and only if the embedding problem $(\tilde \eta^*,\pi_{m+k,m})$ is solvable.
\end{proposition}
The proof is based on the following lemma:
\begin{lemma}\label{basic.lem} \begin{enumerate}
\item For $0\leq k\leq n$, $f\in I^k\hat V_n$ if and only if $f(I^{n-k}V_n)=0$. In particular, the image of the dual map $\pi_{n,m}^*:\hat V_m \ra \hat V_n$ is $I^{n-m}\hat V_n.$
\item The module $\hat V_n$ is cyclic, hence $\hat V_n \cong V_n$. Moreover, an element $f\in \hat V_n$ generates $\hat V_n$  if and only if $f(I^{n-1}V_n)\neq 0$.
\end{enumerate}
\end{lemma}
\begin{proof} Let $R_{i}:=\{f\,|\, f(I^{i}V_n)=0\}$, $0\leq i\leq n$. Note that since $\dim_{\ff{\ell}} I^iV_n=n-i$, one has $\dim_{\ff{\ell}} R_i=i$ for $0\leq i\leq n$.

Fix a generator $\sigma$ of $Z$. If $f=g^{(\sigma-1)^k}$ for $g\in \hat V_n$, then $f(I^{n-k}V_n)=g(I^nV_n)=0$. Hence $I^k\hat V_n\subseteq R_{n-k}$. Applying the dimension formula to the linear transformation $(\sigma-1)^{k}:\hat V_n \ra \hat V_n$ given by $x\ra (\sigma-1)^k x$, one has: $$\dim_{\ff{\ell}}I^{k}\hat V_n = \dim_{\ff{\ell}} \hat V_n - \dim_{\ff{\ell}}\{f\,|\, f^{(\sigma-1)^{k}}=0\} = n-\dim_{\ff{\ell}}R_{k} = \dim_{\ff{\ell}}R_{n-k}. $$
Hence $R_{n-k}= I^k\hat V_n$. The second assertion in Part (a) follows since $f\in \im\pi_{n,m}^*$ if and only if $f(I^mV_n)=0$.

Since the dimension of $\hat V_n$ is $n$, $\ff{\ell}[[Z]]f=\hat V_n$ if and only if the sequence
$$ \ff{\ell}[[Z]]f \supset If \supset \ldots \supset I^{n-1}f\supset I^nf=0$$ is strictly descending. The latter condition holds if and only if $I^{n-1}f\neq 0$ or equivalently $f(I^{n-1}V_n)\neq 0$. 
\end{proof}

\begin{proof}[Proof of Proposition \ref{height.prop}] 
Let $n:=m+k$. The $Z$-embedding problem  $(\tilde\eta^*,\pi_{n,m})$ has a solution  $\psi\colon M\ra V_n$ if and only if its dual $\psi^*\colon \hat V_n\ra \hat M$  satisfies $\pi_{n,m}^*\circ\psi^*=\tilde\eta^*$, i.e. makes the following diagram commutative:
\begin{equation}\label{dual.diag}\xymatrix{
            & \hat M \\
 \hat V_n \ar@{-->}[ur]^{\psi^*} & \hat V_m \ar[l]_{\pi_{n,m}^*}\ar[u]_{\tilde \eta^*},  \\
}
\end{equation}

For the ``if" implication assume there is a solution $\psi:M\ra V_n$. By Lemma~\ref{basic.lem}.(a), we have:  $$\eta\in \im \tilde \eta^* = \im\psi^*\circ\pi_{n,m}^* = \psi^*(I^{k}\hat V_m)=I^{k}\im \psi^*\subseteq I^k\hat M.$$
For the converse assume $\eta=(\sigma-1)^k\eta_n$ for some $\eta_n\in \hat M$. Denote  $f_m:=(\tilde\eta^*)^{-1}(\eta)$. As $f_m$ is a generator of $\hat V_n$, it satisfies  $f_m^{(\sigma-1)^{m-1}}\neq 0$. By Lemma \ref{basic.lem}.(a), there is an $f_n\in \hat V_n$  such that $f_n^{(\sigma-1)^k}=\pi_{n,m}^*(f_m)$. Since $$f_n^{(\sigma-1)^{n-1}}=\pi_{n,m}^*(f_m^{(\sigma-1)^{m-1}})\neq 0,$$ Lemma \ref{basic.lem}.(b) implies that $f_n$ generates $\hat V_n$. Since in addition $I^n\eta_n=0$, we may define $\psi^*\colon \hat V_n\ra \hat M$ to be the unique $Z$-homomorphism for which  $\psi^*(f_n)=\eta_n.$ Then
$$ \psi^*\circ\pi_{n,m}^*(f_m)=\psi^*(f_n^{(\sigma-1)^k})=\eta_n^{(\sigma-1)^k}=\eta=\tilde\eta(f_m). $$ Since $\psi^*\circ\pi_{n,m}^*$ and $\tilde\eta$ agree on a generator of $\hat V_n$, they coincide. Hence $\psi=(\psi^*)^*$ is a solution of $(\tilde \eta^*,\pi_{n,m})$, as required.
\end{proof}

Following Proposition \ref{height.prop}, we define {\bf the height} $\Ht(\phi)$ of a $Z$-homomorphism $\phi\colon M\ra V_m$ to be the maximal $k$ for which $(\phi,\pi_{m+k,m})$ is solvable if such a $k$ exists, and $\infty$ otherwise. Note that by Proposition \ref{height.prop}, for $\eta\in \hat M$, $\Ht(\eta)=\Ht(\tilde \eta^*)$.

Also note that an element $\eta\in \hat M^Z$  is a $Z$-homomorphism. By identifying $\ff{\ell}$ with $V_1$, we may choose $\tilde\eta^*$ to be the dual map of $\eta$. Hence, the height of such $\eta$ as a $Z$-homomorphism and its height as an element of $\hat M$ coincide.

\subsection{A local global principle}
In view of  Propositions \ref{bounded.prop} and \ref{height.prop},
 the finite direct summands of $\oline F$ can be computed using $Z$-embedding problems of the form $(\phi\colon F\ra V_n,\pi_{n,m}\colon V_n\ra V_m).$ To determine the solvability of such embedding problems, we first establish a local global principle.

For a prime $\fp$ of $L$,  let  $Z_\fp$ be the local Galois group $\Gal(L_\fp/\lSylow{K_\fp})$.   Since $\gal(L_\fp)$ is a $Z_\fp$-group, the restriction $(\phi_\fp\colon \Gal(L_\fp)\ra V_n,\pi_{n,m})$ of $(\phi,\pi)$ to $L_\fp$ is a $Z_\fp$-~embedding problem. Furthermore, if $\psi\colon \gal(L)\ra V_n$ is a solution of $(\phi,\pi_{n,m})$ then the restriction $\psi_\fp\colon \gal(L_\fp)\ra V_n$ is a solution of $(\phi_\fp,\pi_{n,m})$ for every prime $\fp$ of $L$. We claim that the converse also holds:
\begin{proposition}\label{local-global.prop} A $Z$-embedding problem $(\phi\colon \gal(L)\ra V_{m},\pi_{n,m})$ is solvable if and only if $(\phi_\fp,\pi_{n,m})$ is solvable for every prime $\fp$ of $L$.
In particular, $\Ht(\phi)=\min_\fp\Ht(\phi_\fp)$ where $\fp$ runs over all primes of $L$.
\end{proposition}
\begin{proof}By Proposition \ref{Z-Frattini.prop}, there is a global field $K(\mu_\ell)\leq K'\leq \lSylow{K}$ such that $\phi$ extends to a $Z$-homomorphism $\phi'\colon \gal(L')\ra V_{m}$, where $L':=K'(\mu_{\ell^\infty})$. We identify $Z=\gal(L/\lSylow{K})$ and $\gal(L'/K')$ via the restriction map. For every prime $\fp$ of $L$, this gives an identification of $Z_\fp$ with the decomposition group of $\fp\cap L'$ in $L'/K'$. %

 Let $A:=\ker \pi_{n,m}$.
Then $A$ is  a $\Gal(K')$-module via the restriction $\gal(K')\ra Z$. We claim
that the map:
$$\rho:H^2(\gal(K'),A)\ra \prod_{\fp} H^2(\gal(K'_\fp),A)$$ is injective, where $\fp$ runs over all primes of $K'$.
Let $\hat A=\Hom(A,\mu_\ell)$ be the dual $\gal(K')$-module with the action $f^\sigma(x)=f(x^{\sigma^{-1}})^{\sigma}$ for $\sigma\in \gal(K')$, $x\in A$, and $f\in \hat A$. Let $K'(\hat A)$ be the fixed field of the centralizer $H\leq \Gal(K')$ of $\hat A$ under the action of $\gal(K')$. Since $\gal(K')$ acts trivially on $\mu_\ell$ and $\gal(L')$ acts trivially on $A$, the map $\gal(K')\ra\Aut(\hat A)$ splits through $Z\cong \gal(L'/K')$. Thus, $H$ is an open subgroup of $\gal(K')$ which contains $\gal(L')$, and hence $G':=\Gal(K'(\hat A)/K')$ is a finite cyclic $\ell$-group as a quotient of $Z$.
 By the Poitou-Tate duality theorem \cite[Satz 4.5]{Neu2} (or \cite[Theorem~8.6.8]{NSW}),
 $\rho$ is injective if and only if
$$\rho'\colon H^1(G',\hat A)\ra \prod_\fp H^1(G_\fp',\hat A)$$ is injective, where $\fp$ runs over all primes of $K'$. Here $G_{\fp}'=\Gal(K'(\hat A)_\frak{P}/K'_\fp)$ for some prime $\frak{P}$ of $K'(\hat A)$ lying over $\fp$.
Since $G'$ is cyclic, by Chebotarev's density theorem there are infinitely many primes $\fp$ for which $G'_\fp=G'$. Thus, $\rho'$ and hence $\rho$ are injective, as claimed.

Let $\tilde\phi\colon \Gal(K')\ra  V_{m}\rtimes Z$ be the map given by the composition of the isomorphism $\gal(K')\cong \gal(L')\rtimes Z$ and the map $(\phi',\id)\colon \gal(L')\rtimes Z \ra V_m\rtimes Z$, and let $\tilde{\pi}_{n,m}\colon V_n\rtimes Z\ra V_m\rtimes Z$ be the map defined by $\tilde{\pi}_{n,m}(x,z)=(\pi_{n,m}(x),z)$.
Since the $Z$-embedding problem $(\phi',\pi_{n,m})$ is solvable if and only if  the embedding problem $(\tilde\phi, \tilde{\pi}_{n,m})$ is solvable, it suffices to show the latter.
Similarly, since $(\phi_\fp,\pi_{n,m})$ is solvable, the restriction
$(\tilde{\phi}_\fp, \tilde{\pi}_{n,m})$ of $(\tilde{\phi},\tilde{\pi}_{n,m})$ to $\gal(K'_\fp)=\Gal(L'_\fp)\rtimes Z_\fp$ is solvable.
The maps $\tilde{\phi},\tilde{\phi_\fp}$  form the following commutative diagram:
\begin{equation}\label{res-inf.equ}
\xymatrix{
\HLG^2(V_m\rtimes Z,A)  \ar[r]^<<<<<{\tilde \rho} \ar[d]_{\tilde{\phi}^*} & \prod_\fp \HLG^2(V_m\rtimes Z_\fp,A) \ar[d]^{\prod_\fp\tilde{\phi}_\fp^*} \\
\HLG^2(\gal(K'),A) \ar[r]^<<<<{\rho} & \prod_\fp \HLG^2(\gal(K'_\fp),A),
}
\end{equation} where $V_m\rtimes Z$ acts on $A$ via the projection onto $Z$,  $\tilde \rho$ is the restriction map, and $\fp$ runs through all primes of $L'$.

Since the action of $V_m\rtimes Z$ on $A$ via the extension $\tilde\pi_{n,m}$ factors through the projection onto $Z$, it agrees with the above chosen action. Let $\alpha_{n,m}\in \HLG^2(V_m\rtimes Z,A)$ be the class defined by $\tilde\pi_{n,m}$, and  $\alpha_{n,m}^{(\fp)}$ be the $\fp$-th component of $\tilde \rho(\alpha_{n,m})$. Since $(\tilde\phi_\fp,\tilde\pi_{n,m})$ is solvable, $\tilde\phi^*_\fp(\alpha_{n,m}^{(\fp)})=0$ for all $\fp$. By (\ref{res-inf.equ}), $\rho\circ\tilde\phi^*(\alpha_{n,m})=0$. Since $\rho$ is injective, $\tilde\phi^*(\alpha_{n,m})=0$ and hence $(\tilde{\phi},\tilde\pi_{n,m})$ is solvable, as required.
\end{proof}

\subsection{The local height}\label{local-height.sec}
The above local global principle reduces the computation of the global height $\Ht(\phi)$ of a $Z$-homomorphism $\phi\colon  F \ra V_m$, to the computation of the local heights $\Ht(\phi_\fp)$ for all primes $\fp$ of $L$. We compute the latter using Iwasawa theory \cite{Iw}.

A homomorphism $\phi\colon \gal(L)\ra G$ is {\bf unramified} (resp. {\bf tamely ramified}) at a prime $\fp$ of $L$ if the fixed field of $\ker(\phi)$ is unramified (resp. tamely ramified) over $L$  at $\fp$.
\begin{proposition}\label{local-height.prop} Let $\fp$ be a prime of $L$ and $\ell^t:=[Z:Z_\fp]$. Let $\phi\colon  F\ra V_m$ a $Z$-homomorphism.   Then:
 \begin{enumerate}
 \item Either $\Ht(\phi_\fp)=\infty$ or $\ell^t-m \leq \Ht(\phi_\fp)< \ell^t$;
 \item If $\phi$ is unramified, then $\Ht(\phi_\fp)=\infty$;
 \item If $\phi$ is ramified nontrivially and tamely, then $\ell^t-m \leq \Ht(\phi_\fp)< \ell^t$.
 \end{enumerate}
\end{proposition}
\begin{proof} 
 If $\fp$ is infinite, $\fp$ is complex since $L$ contains all $\ell$-power roots of unity. Hence for infinite $\fp$, $\phi_\fp$ is trivial and $\Ht(\phi_\fp)=\infty$.

Assume $\fp$ is a finite prime.
By Proposition \ref{Z-Frattini.prop}, $\phi$ extends to a $Z$-homomorphism $\phi'\colon \gal(L')\ra V_m$, where  $L'=K'(\mu_{\ell^\infty})$ and $K'/K(\mu_\ell)$ is a finite extension.
Moreover,  $\Ht(\phi_\fp)=\Ht(\phi_{\fp\cap L'}')$ for any prime $\fp$ of $L$.
Let $G:=\Gal(L'_\fp)$ and $\ab{G}$ (resp. $\oline G$) the maximal abelian (resp. elementary abelian) quotient of $G$ viewed as $Z_\fp$-groups.



Iwasawa's theorem \cite[Theorem 25]{Iw} gives a $Z_\fp$-isomorphism
$s\colon \ab{G} \ra T(\mu)\times \Lambda^d,$
where $T(\mu)$ is the Tate module  $T(\mu):=\varprojlim\mu_{\ell^n}$,
$\Lambda:=\mathbb{Z}_\ell[[Z_\fp]]$, 
and $d=[K'_\fp\colon \mQ_\ell]$ if $\fp$ lies over $\ell$ and $0$ otherwise.
Moreover, $s^{-1}$ is obtained as an inverse limit of the reciprocity maps $r_E:E^\times \ra \ab{\gal(E)}$ where $E$ runs through finite intermediate extensions $K'\subseteq~E\subseteq~L'$, see \cite[End of Pg. 319]{Iw}. Since $r_E$ maps the units of $E$ to the inertia subgroup of $\ab{\gal(E)}$, the inverse limit $T(\mu)$ of $\ell$-power roots of unity is mapped under $s^{-1}$ to the inertia subgroup of $\ab{G}$.

As $\Lambda/\ell\Lambda \cong \ff{\ell}[[Z]]$ and $T(\mu)/\ell T(\mu)\cong V_1$ as $Z_\fp$-modules, $s$ gives a $Z_\fp$-isomorphism
$$\oline G=\ab{G}/\ell\ab{G}\cong  V_1\times \ff{\ell}[[Z]]^d.$$ Let $G_1$ be the direct $Z$-summand of $\oline G$ which corresponds to $V_1$ under this isomorphism. Hence, $G_1$ is contained in the inertia subgroup of $\oline G$.

We separate into two cases as to whether $G_1$ is contained in $\ker\phi_\fp'$.  If $G_1\leq~\ker\phi_\fp'$, then $\phi'_\fp$ splits through $\ff{\ell}[[Z]]^d$. As $\ff{\ell}[[Z]]^d$ is free as an $\ff{\ell}[[Z]]$-module, the embedding problem $(\phi'_\fp,\pi_{n+m,m})$ is solvable for all $n\in\mathbb{N}$. Thus,  $\Ht(\phi_\fp)=\Ht(\phi_\fp')=\infty$. This is in particular the case if $\phi'_\fp$ is unramified, proving (b).

On the other hand if $G_1\not\leq \ker\phi'_\fp$, we claim that $\ell^t-m \leq \Ht(\phi_\fp')< \ell^t$. To show that $\ell^t-m \leq \Ht(\phi_\fp')$, it suffices to show that $(\phi'_\fp,\pi_{n,m})$ is solvable if $n-m= \ell^t-m$, that is, $n= \ell^t$.   Let  $\sigma$ be a generator of $Z$. Since $(\sigma^{\ell^t}-1)=I^{\ell^t},$ and since $[Z:Z_\fp]=\ell^t$, the $Z$-module $V_{\ell^t}$  is the trivial $Z_\fp$-module $(\ff{\ell})^n$. In particular, the $Z_\fp$-embedding problem $(\phi'_\fp,\pi_{\ell^t,m})$ is solvable, as claimed.

To show $\Ht(\phi_\fp')< \ell^t$, assume  $n-m= \ell^t$, that is, $n=m+\ell^t$. Furthermore, assume on the contrary  that  $(\phi'_\fp,\pi_{n,m})$ is solvable. Hence, its restriction $$(\phi''_\fp\colon G_1\ra V_m,\pi_{n,m})$$ to $G_1$ has a solution, say $\psi_\fp$.
Since $G_1$ is fixed by $Z_\fp$ so is its image $J:=\im\psi_\fp$. Thus, $I^{\ell^t}J=(\sigma^{\ell^t}-1)J=0.$
Since the kernel of the map $V_n\ra V_n, x\ra x^{\sigma^{\ell^t}-1}$  is $I^mV_n$, we have $J\subseteq I^mV_{n}=\ker\pi_{n,m}$. Hence, $\im(\pi_{n,m}\circ \psi_\fp)=\im(\phi''_\fp)=\{0\}$. But $\im(\phi''_\fp)\neq~0$ since $G_1\not\subseteq \ker\phi_\fp'$. This contradiction proves the claim and Part (a).

If $\phi_\fp$  ramifies nontrivially and tamely, $\fp$ does not divide $\ell$, so $d=0$ and $\oline G=G_1$. As $\phi_\fp$ is nontrivial, this implies that $G_1\not\leq \ker\phi'_\fp$. In this case, the above claim gives Part~(c), completing the proof.
\end{proof}
For $m=1$ we get:
\begin{corollary}\label{height1.cor}
Let $\fp$ be a prime of $L$ and $\phi\colon F\ra V_1$ a $Z$-homomorphism. Then $\Ht(\phi)=[Z:Z_\fp]-1$ or $\infty$. If $\phi$ is unramified then $\Ht(\phi)=\infty$. If $\phi$ is ramified nontrivially and tamely then $\Ht(\phi)=[Z:Z_\fp]-1$.
\end{corollary}

\subsection{Finite Ulm invariants}
The following proposition gives the finite Ulm invariants of $\hat F$, and hence in view of Proposition \ref{bounded.prop} the finite direct $Z$-summands of $\hat F$. Its proof combines the above local global principle and computation of local heights.
\begin{proposition}\label{finite-Ulm.prop}
The $n$-th Ulm invariant of $\hat F$ is:
$$ U_n(\hat F) = \left\{ \begin{array}{ll} \omega & \mbox{ if } n=\ell^k-1 \mbox{ for }k\in\mathbb{N}\cup\{0\} \\ 
0 & \mbox{ for any other }n\in\mathbb{N} \end{array} \right.$$
\end{proposition}
\begin{proof}
Since an element   $\eta\in \hat F^Z$ is a $Z$-homomorphism, its height is the maximal $n$ such that $(\eta,\pi_{n+1,1})$ is solvable.
 Thus, Proposition \ref{local-global.prop} and  Corollary \ref{height1.cor} imply that the height of each element of $\hat F^Z$ is either infinite or $\ell^k-1$, for some $k$. Hence, by \eqref{finite-Ulm.equ}, $U_n(\hat{F})=0$ for all other~$n\in\mathbb{N}$.

For $n=\ell^k-1$, $k\in\mathbb{N}\cup \{0\}$, we shall construct an infinite subgroup $F_n\leq \hat F^Z$, the nontrivial elements of which are of height $\ell^k-1$.

Let $\ell^s$ be the number of $\ell$-power roots of unity in $K(\mu_\ell)$ and hence in $\lSylow{K}$. 
We first claim that there exists an infinite set $P_k$ of rational primes $p$ such that
$p~\equiv~1 \Mod{\ell^{k+s}}$,
$p\not\equiv 1\Mod{\ell^{k+s+1}}$,
and such that there is a prime $\fq$ of $K$ of degree one over $p$. 

Let $M$ be the Galois closure of $K/\mQ$ and let $C\leq \Gal( M(\mu_{\ell^{k+s+1}})/K(\mu_{\ell^{k+s}} ))$ be a cyclic subgroup which does not fix $\mu_{\ell^{k+s+1}}$. By Chebotarev's density theorem there are infinitely many rational primes $\fq'$ of $M(\mu_{\ell^{k+s+1}})$  whose Frobenius lies in~$C$. Since $C$ fixes $K$, the restriction $\fq$ of such $\fq'$ to $K$ is of  degree one over $(p)=\fq'\cap\mQ$. Since the restriction of $C$ to $\mQ(\mu_{\ell^{k+s+1}})$ lies in $\Gal(\mQ(\mu_{\ell^{k+s+1}})/\mQ(\mu_{\ell^{k+s}}))$, we get that $p\equiv 1\Mod{\ell^{k+s}}$ and $p\not\equiv 1\Mod{\ell^{k+s}}$, proving the claim.

For each $p\in P$, let $\phi'_p:\gal(\mQ)\ra \ff{\ell}$ be a nontrivial homomorphism ramified only over $p$, and $\phi_p\in \hat{F}^Z$ be its restriction to $F$. Let $F_n$ be the subgroup of $\hat F$ generated by  $\phi_p$, $p\in P$.

We claim that every nontrivial $\phi\in F_n$ is of height $\ell^k-1$.
In view of Proposition~\ref{local-global.prop}, it suffices to consider the local heights.
Since $\phi'_p$ is ramified only over $p$, $\phi$ is ramified only over primes of $L$ lying over primes in~$P$. Since $p\equiv 1\Mod{\ell^{k+s}}$ for every $p\in P$, one has $\mu_{\ell^{k+s}}\subseteq \mathbb{Q}_{p}\subseteq L_\fp$, and hence $\ell^k\divides [Z:Z_\fp]$ for every prime $\fp$ of $L$ dividing $p$. Thus by Corollary~\ref{height1.cor}, $\Ht(\phi_\fp)\geq \ell^k-1$ for all primes $\fp$ of $L$.

Since $\phi$ is the restriction of a nontrivial linear combination of $\phi'_p$, $p\in P$,
there is a prime $q\in P$ such that $\phi$ is ramified over all primes of $L$ dividing $q$.
Let $\fq_0$ be a degree one prime of $K$ over $q$. Thus,  $\phi$ is ramified  over a prime $\fQ_0$ of $\lSylow{K}$ lying over $\fq_0$.   Since $\mu_{\ell^{k+s+1}}\not\subseteq \mQ_{q}\cong K_{\fq_0}$, we have $\mu_{\ell^{k+s+1}}\not\subseteq \lSylow{K_{\fQ_0}}$ and hence $[Z:Z_{\fQ_0}]=\ell^k$. By Corollary~\ref{height1.cor},  $\Ht(\phi_{\fQ_0})=\ell^k-1$.
It therefore follows from Proposition~\ref{local-global.prop} that $$\Ht(\phi)=\min_{\fp} \Ht(\phi_\fp)=\Ht(\phi_{\fQ_0})=\ell^k-1,$$ for every $\phi\in F_n$, proving the claim. By \eqref{finite-Ulm.equ}, we get $U_{\ell^k-1}(\hat{F})=~\omega$, for all nonnegative integers $k$. 
\end{proof}

\subsection{Proof of Theorem \ref{thm:direct-summands}}

We shall deduce the finite direct summands of $\oline F$ directly from Propositions \ref{bounded.prop} and \ref{finite-Ulm.prop}. The following lemma describes the only possible infinite indecomposable summands.
\begin{lemma}\label{indecomposable.lem} Let $P$ be a discrete countable indecomposable torsion $\ff{\ell}[[Z]]$-module. Then either $P\cong V_n$ for some $n\in\mathbb{N}$, or  $P\cong \hat V$ where $V:=\ff{\ell}[[Z]]$.
\end{lemma}
\begin{proof} If $U_n(P)\neq 0$ for some natural number $n$, then $\hat V_n$ is a direct summand of $P$ by Proposition \ref{bounded.prop}. As $P$ is indecomposable it follows that in such case $P\cong \hat V_n\cong V_n$. Thus, we may assume that $P$ has trivial finite Ulm invariants. Such $P$ satisfies $IP=P$, i.e. it is  a divisible $\ff{\ell}[[Z]]$-module. By \cite[Theorem~4]{Kap}\footnote{As noted in \cite[\S 12]{Kap} the proof of \cite[Theorem 4]{Kap} for $\mathbb{Z}$-modules also holds for $\ff{\ell}[[Z]]$-modules when replacing the $\ell$-primary part $\mQ_\ell/\mZ_\ell=\indlim Z/\ell^nZ$ of $\mQ/\mZ$ by $\hat V=\indlim \hat V_n$. } 
every divisible $\ff{\ell}[[Z]]$-module is isomorphic to a direct sum of $\ff{\ell}[[Z]]$-modules isomorphic to $\hat V$. Thus, if $P$ is divisible and indecomposable $P\cong \hat V$.
\end{proof}

The proof of Theorem \ref{thm:direct-summands} therefore reduces to finding the multiplicity of $\hat V$ as an $\ff{\ell}[[Z]]$-summand of $\hat F$, or equivalently the multiplicity of $V$ as an $\ff{\ell}[[Z]]$-summand of $\oline F$. This is done using the following proposition. Note that the dual of the maximal divisible $\ff{\ell}[[Z]]$-submodule of $\hat F$ is the maximal free $\ff{\ell}[[Z]]$-quotient of~$\oline F$.

\begin{proposition}\label{free-quotient.prop} Let $K$ be a global field. Then the maximal free $\ff{\ell}[[Z]]$-quotient of $\oline F$ is   $\ff{\ell}[[Z]]^\omega$ if $\chr K = 0$,
and is trivial if  $\ell\neq\chr K>0$. 
\end{proposition}
\begin{proof}
First assume that $K$ is a number field. Let $K(\mu_\ell)\subseteq K'\subseteq \lSylow{K}$ be a number field. By Iwasawa theory \cite[Theorem 13.31]{Wa} there is a $Z$-homomorphism $$\gal(K'(\mu))\ra \Lambda^{r_2(K')}$$ with finite cokernel, where $\Lambda:=\mathbb{Z}_\ell[[Z]]$. Let $J$ be its image.  Since $J/\ell J$ is an $\ff{\ell}[[Z]]$-submodule of finite index in $(\Lambda/\ell\Lambda)^{r_2(K')}\cong \ff{\ell}[[Z]]^{r_2(K')}$ and $\ff{\ell}[[Z]]$ is a discrete valuation ring, $J/\ell J$ is $\ff{\ell}[[Z]]$-isomorphic to $\ff{\ell}[[Z]]^{r_2(K')}$. This shows that $\ff{\ell}[[Z]]^{r_2(K')}$ is a $Z$-quotient of $\gal(K'(\mu_{\ell^\infty}))$ and hence, by Lemma \ref{restriction.lem}, it is also a $Z$-quotient of $\gal(L)$.
Since $r_2(K')$ is arbitrarily large for prime to-$\ell$ extensions we get the desired result in case $\chr K=0$.

Assume $\ell\neq \chr K>0$. It suffices to show that the $Z$-embedding problem $(\phi,\pi_1\colon V\ra V_1)$ is nonsolvable for every $Z$-homomorphism $\phi\colon F \ra V_1$.
By Proposition \ref{Z-Frattini.prop}, $\phi$ extends to a $Z$-homomorphism $\phi'\colon \Gal(L')\ra V_1$, where $L'=K'(\mu_{\ell^\infty})$ for some finite subextension $K'$ of $\lSylow{K}/K(\mu_\ell)$.   By \cite[\S 12.4]{Iw}, the maximal abelian $Z$-quotient $X:=\gal(L')^{ab}$ is a $\Lambda$-torsion module for which $X/\ell X$ has   no free $\Lambda/\ell\Lambda\cong \ff{\ell}[[Z]]$-quotients. Thus, $(\phi',\pi_1)$ is nonsolvable. Hence, by Proposition~\ref{Z-Frattini.prop}, $(\phi,\pi_1)$ is nonsolvable, as required.
\end{proof}



\begin{proof}[Proof of Theorem \ref{thm:direct-summands}] By Lemma \ref{indecomposable.lem} it suffices to find  the multiplicities of $V_n$ and $\hat V$ as summands of $\oline F$. By Propositions \ref{bounded.prop} and \ref{finite-Ulm.prop}, the multiplicity of $V_n$ is $\omega$ if $n=\ell^k$ for $k\in \mathbb{N}\cup\{0\}$, and $0$ otherwise. Note that for a generator $\sigma$ of $Z$,  $(\sigma-1)^{\ell^k}=\sigma^{\ell^k}-1$. Thus,  $I^{\ell^k}=(\sigma^{\ell^k}-1)$ and hence 
$V_{\ell^k}\cong \ff{\ell}[Z/\ell^kZ]$, 
for every $k\in\mathbb{N}\cup\{0\}$. Thus, $\ff{\ell}[Z/\ell^k Z]$ is a direct summand of $\oline F$ with multiplicity $\omega$.

Since $\ff{\ell}[[Z]]$ is a free $\ff{\ell}[[Z]]$-module, the maximal free $\ff{\ell}[[Z]]$-quotient of $\oline F$ is its direct summand. Proposition \ref{free-quotient.prop} then implies that $\ff{\ell}[[Z]]$ has multiplicity $\omega$ in $\oline F$. \end{proof}

\begin{corollary}\label{action-mod-phi.cor}
For any positive integer $N$ the $Z$-group $\overline{F}$ decomposes as $\overline{F} = F_{\leq N} \times F_{>N}$ where:
$$F_{\leq N}\cong \ff{\ell}[[Z]]^\kappa \times \prod_{k=0}^N \ff{\ell}[Z/\ell^kZ]^\omega,$$
 $\kappa=\omega$ if $K$ is a number field and $\kappa=0$ otherwise,
 and $F_{>N}$ has no $\ff{\ell}[[Z]]$-summands of dimension $\leq \ell^N$ over $\ff{\ell}$, nor $\ff{\ell}[[Z]]$-summands isomorphic to $\ff{\ell}[[Z]]$.
\end{corollary}
\begin{proof}
As in Theorem \ref{thm:direct-summands}, Proposition \ref{bounded.prop} gives a decomposition $\oline F=V_{\leq N}\times V_{>N}$, where
 $$V_{\leq N}\cong  \prod_{0\leq k\leq N}\ff{\ell}[Z/\ell^kZ]^{\omega},$$ 
 and $V_{>N}$  has  no direct $\ff{\ell}[[Z]]$-summands of dimension $\leq \ell^N$. If $\ell\neq \chr K>0$, this is the desired decomposition.

 If $K$ is a number field,  $\ff{\ell}[[Z]]^\omega$ is a quotient of $\oline F$, and hence of $V_{>N}$. Furthermore, since  $\ff{\ell}[[Z]]^\omega$ is free, it is a direct summand of $V_{>N}$. Letting $F_{\leq N}$ be the product of $V_{\leq N}$ and the $\ff{\ell}[[Z]]^\omega$ summand of $V_{>N}$, and letting $F_{>N}$ be a complement of the latter summand in $V_{>N}$, we obtain the desired decomposition.
\end{proof}
%


\subsection{Towards a presentation}


As a Corollary to Theorem \ref{action-mod-phi.cor}, we get the following description of $\gal(\lSylow{K})$ in terms of generators and relations.

Let $\sigma$ be a generator of $Z$ and let $x^\sigma=\sigma^{-1} x\sigma$ denote the action of $\sigma$ on $x\in F$. Recall that $X\subseteq F$ is a basis for $F$ if $X$ converges to $1$, and $F$ is the free pro-$p$ group generated by $X$ \cite[\S 3.3]{RZ}.
\begin{corollary}\label{generators.cor} Assume $K$ be a number field, and  $N$ a positive integer. Then $\gal(\lSylow{K})$ is generated by $\sigma$  and a basis of $F$ which
 is a disjoint union of three subsets  $X_{>N}\cup X_\infty \cup X_{\leq N}$:
\begin{enumerate}
\item  $X_{\leq N}$ is a disjoint union of infinitely many copies of each of the sets
$$\{x_0,\ldots, x_{\ell^n-1}\}, n\leq N,$$ subject  to the relations \begin{equation}\label{cyclic-relation.equ}x_i^\sigma = x_{i+1}y_i \text{ and } x_{\ell^n-1}^\sigma=x_{0}y
\end{equation}
  for some $y,y_i\in \Phi(F)$, $0\leq i\leq \ell^n-2$;
\item $X_{\infty}$ is a disjoint union of infinitely many copies of  the set $\{x_n\}_{n=0}^\infty$
which converges to $1$  as $n\ra\infty$, and is subject to the relations:
\begin{equation}\label{free-relation.equ}x_i^\sigma= x_{i+1}x_iy_i \text{ and } x_1^{\sigma^{-1}}=(\prod_{i=0}^\infty x_i^{(-1)^i})y,\end{equation}
for some $y,y_i\in\Phi(F)$,  $i\in\mathbb{N}\cup\{0\}$;

\item  $\langle X_{>N},\Phi(F)\rangle$ is $Z$-invariant.
\end{enumerate}
Moreover, we can assume that any finite subset of the $y_i$'s appearing in parts (b) and (c) are trivial.
\end{corollary}
\begin{proof}
Recall that a basis for $\oline F$ as a profinite $\ff{\ell}$-vector space is a minimal generating set which converges to $1$. We first choose a basis $\oline S$ for $\oline F$ using the decomposition in Corollary \ref{action-mod-phi.cor} as follows. For each $\ff{\ell}[[Z]]$-summand isomorphic to $V_{\ell^n}\cong \ff{\ell}[Z/\ell^nZ]$, $n\leq N$,  include in $\oline S$  the basis $\{\oline x_i\}_{i=0}^{\ell^n-1}$ of the summand which corresponds to the basis $\sigma^i$, $i=0,\ldots,\ell^n-1,$ of $\ff{\ell}[Z/\ell^nZ]$. For each $\ff{\ell}[[Z]]$-summand isomorphic to $\ff{\ell}[[Z]]$, include a basis $\{\oline x_i\}_{i=0}^{\infty}$ which corresponds to $(\sigma-1)^i$, $i=0,1,\ldots$. Include in $\oline S$ a basis of $V_{>N}$. Note that since each of the above bases converges to $1$, their union $\oline S$ converges to $1$ in the product topology. Hence the set $\oline S$ is a basis for $\oline F$. 

By Burnside's basis theorem \cite[Proposition 7.6.9]{RZ}, a basis $\oline S$ for $\oline F$ can be lifted to basis $S$ of $F$. Since for each $V_{\ell^n}$-summand we have $\oline x_{i+1}=\oline x_i^{\sigma}$, $i=0,\ldots, \ell^n-2$, and $\oline x_{\ell^n-1}^{\sigma}=\oline x_1$, the relations in \eqref{cyclic-relation.equ} follow. The relations in \eqref{free-relation.equ} follow since for each $\ff{\ell}[[Z]]$-summand we have $\oline x_{i+1}=\oline x_i^{\sigma-1}:=\oline x_i^{\sigma} - \oline x_i$, $i=0,\ldots$, and $$\oline x_1^\sigma=\sum_{i=0}^\infty  \oline x_1^{(1-\sigma)^i} =\sum_{i=0}^\infty (-1)^i \oline x_i. $$

Moreover, by \cite[Corollary 7.6.10]{RZ} the basis $\oline S$ can be lifted to a basis $S$ of $F$ in which finitely many elements in $\oline S$ have prescribed liftings. Thus, we may assume that finitely many of the $y_i$'s in Parts (b) and (c) equal $1$.
\end{proof}


\subsection{Infinite Ulm invariants}\label{infinite-ulm.sec}
To completely determine the structure of $\oline F$ as a $Z$-module, it remains to find the infinite Ulm invariants of $\hat F$ or equivalently the Ulm invariants of $I^\omega \hat F$. The latter relates to Iwasawa modules as follows.

Let $M$ be the maximal abelian pro-$\ell$ extension of $K(\mu_{\ell^\infty})$ unramified away from primes dividing $\ell$, and $M^{\Un}$ the maximal subfield of $M$ which is unramified over $K(\mu_{\ell^\infty})$. Iwasawa theory \cite[\S 13]{Wa} studies the Galois groups $X^{\Un}(K):=\Gal(M^{\Un}/K)$ and $X(K):=\Gal(M/K)$ as modules over $\gal(K(\mu_{\ell^\infty})/K)$.

\begin{proposition}\label{prop:infinite} Let $K$ be a global field, $X:=X(\lSylow{K})$ and $X^{\Un}:=X^{\Un}(\lSylow{K})$. Then $\hat{X}^{\Un}\subseteq I^\omega\hat F \subseteq \hat X.$
\end{proposition}
The proof is based on the following lemma. As in Proposition \ref{height.prop}, for $\eta\in \hat F$, let $\tilde\eta:\hat V_n \ra \hat F$ be an $\ff{\ell}[[Z]]$-monomorphism whose image is $\ff{\ell}[[Z]]\eta$, and  $\tilde\eta^*:F\ra V_n$ its dual map.
 \begin{lemma}\label{lem:normal-closure}
Let $\eta\in \hat F$ and $E$ the fixed field of $\ker\eta$.  Then the fixed field of $\ker\tilde\eta^*$ is the normal closure of $E/\lsylow{K}$.
\end{lemma}
\begin{proof}
Let $U:=\ker\eta$, so that $U=\Gal(E)$.  Since every element in $\im\tilde\eta$ is an $\ff{\ell}$-linear combination of  $\eta^{\sigma^i}$, $i=0, \ldots, n-1$, $\im\tilde\eta$ consists of all $\chi\in \hat F$ such that $\chi(\bigcap_{i=0}^{n-1} U^{\sigma^i})=0$. By duality $\ker\tilde\eta^*=\bigcap_{i=0}^{n-1}U^{\sigma^i}$. Thus, the fixed field of $\ker\tilde\eta^*$ is the compositum $M$ of $E^{\sigma^i}$, $i=0,\ldots, n-1$. Since the conjugates of $E$ are contained in the normal closure of $E/\lsylow{K}$, so is $M$. Since $L/\lsylow{K}$ is Galois, $E/L$ is Galois, and since $\sigma$ extends to $M$, $M/\lsylow{K}$ is Galois. Thus, $M$ equals the normal closure of $E/\lsylow{K}$.
\end{proof}

\begin{proof}[Proof of Proposition \ref{prop:infinite}]
Assume $\eta\in \hat F$ is unramified. Since the fixed field of $\eta$ is unramified over $L$, so is its normal closure over $L$. Hence, by Lemma \ref{lem:normal-closure} the map $\tilde\eta^*$ is unramified. By Propositions \ref{local-global.prop} and \ref{local-height.prop}, $\Ht(\tilde\eta^*)=\infty$. Hence by Proposition \ref{height.prop}.(b) one has $\Ht(\eta)=\infty$, proving the first containment.

For the second containment, assume $\Ht(\eta)=\infty$. By Proposition \ref{height.prop}, the map $\Ht(\tilde\eta^*)=\infty$. By Proposition \ref{local-height.prop}.(c), the map $\tilde\eta^*$ is unramified away from primes dividing $\ell$. By Lemma \ref{lem:normal-closure}, $\eta$ is unramified away from primes dividing $\ell$, and hence splits through $\Gal(M/L)$, proving the second containment.
\end{proof}

We next use the structure of the Iwasawa modules $X$ and $X^{\Un}$ to study $I^\omega \hat F$.
Letting $L_0$ be the $\mZ_\ell$-subextension of $K(\mu_{\ell^\infty})/K$, by Lemma \ref{splitting.lem} we may identify $Z$ with $\Gal(L_0/K)$ so that the restriction $\gal(L_0)\ra \gal(L)$ is a $Z$-homomorphism.
By \cite{Iw}, the $Z$-modules $X(K)$ and $X^{\Un}(K)$ are finitely generated and hence admit a  $Z$-homomorphism with finite kernel and cokernel into a unique $Z$-module of the form:
$$\Lambda^{r}\times  \prod_{i\in I}(\Lambda/\ell^i\Lambda)^{r_i} \times \prod_{j=1}^k \Lambda/(g_i(x)),$$
where $\Lambda:=\mathbb{Z}_\ell[[Z]]$,  $I\subseteq \mathbb{N}$ is a finite subset, $r,k,r_i\in \mathbb{N}$ for all $i\in I$,  and $g_j(x),j=1,\ldots, k,$ are monic irreducible polynomials for which all nonleading coefficients are divisible by $\ell$. The Iwasawa {\bf $\mu$-invariant} of such a $Z$-module is the corresponding sum $\sum_{i\in I}r_i$.

\begin{proposition}
Let $K=\mQ$ and $\ell$ an odd prime. Then  $I^\omega \hat{F}$ has nontrivial Ulm invariants.
\end{proposition}
\begin{proof}
We shall construct a $Z$-homomorphism $\phi\colon \gal(L)\ra V_1$ with $\Ht(\phi)=\infty$ and such that $(\phi,\pi_1\colon V\ra V_1)$ is nonsolvable. This will show that   $I^\omega \hat F$ is not a direct sum of $\ff{\ell}[[Z]]$-modules isomorphic to $\hat V$, as otherwise its dual would be a free $\ff{\ell}[[Z]]$-module. Thus by \cite[Theorem 4]{Kap},  $I^\omega \hat F$ is not divisible, and hence  $I^\omega\hat F$ has nontrivial Ulm invariants, as required.

By \cite{Yam}, there exists  a real quadratic extension  $K_0/\mQ$ whose class number is divisble by $\ell$. Hence there is an unramified $\mathbb{Z}/\ell\mathbb{Z}$-extension $M_0/K_0$.  We define $\phi\colon \gal(L)\ra V_1$ as the restriction of  a homomorphism $\phi_0'\colon \gal(K_0)\ra \ff{\ell}$ whose kernel fixes $M_0$. Since $\phi$ is unramified, Proposition \ref{prop:infinite} shows that $\Ht(\phi)=\infty$.

Assume on the contrary that there is a solution $\psi$ to $(\phi,\pi_1)$.
Let $L_0/K_0$ be the $\mathbb{Z}_\ell$-extension inside $K_0(\mu_{\ell^\infty})$, and $\phi_0$ the restriction of $\phi$ to $L_0$.
 By Proposition \ref{Z-Frattini.prop}, $\psi$ extends to a solution $\psi_0$ of the $Z$-embedding problem $(\phi_0,\pi_1)$.
Let $K_1=K_0(\mu_\ell)$, $L_1=K_0(\mu_{\ell^\infty})$ and $\Delta:=\gal(K_1/K_0)$.
In particular, $\phi_0$ splits through a $Z$-homomorphism $\phi_u:X(K_0)/\ell X(K_0)\ra V_1$.

At primes $\fp$ of $L$ that are prime to $\ell$, $\gal(L_\fp)$ is cyclic and in particular has no free $\ff{\ell}[[Z]]$-quotients. Thus, $\psi$ and hence $\psi_0$ are unramified at primes that do not divide $\ell$. It follows that $\psi_0$ factors through $X(K_0)$ and hence through $X(K_0)/\ell X(K_0)$, showing that $(\phi_u,\pi_1)$ is  solvable.

Let  $M^{\Sc}$ be the maximal unramified pro-$\ell$ extension of $K_1(\mu_{\ell^\infty})$ in which all primes dividing $\ell$ split completely. Let $X^{\Sc}(K_1):=\Gal(M^{\Sc}/K_1(\mu_{\ell^\infty}))$.
By Iwasawa's theorem \cite[Corollary 11.3.17]{NSW},   the  $\mu$-invariants  $\mu(X(K_1))$ and $\mu(X^{sc}(K_1))$ are equal. By Ferrero-Washington \cite{FW}, $\mu(X^{\Un}(K_1))=0$. Since $X^{\Un}(K_1)$ has no free $\Lambda$-quotients \cite[Proposition 13.19]{Wa}, $\mu(X^{\Sc}(K_1))\leq \mu(X^{\Un}(K_1))=0$ and hence $\mu(X^{\Sc}(K_1))=\mu(X(K_1))=0$.
 The  module $X(K_1)$  over $\gal(K_1(\mu_{\ell^\infty})/K_0)\cong \Delta\times Z,$ decomposes  into a direct sum $\oplus \varepsilon_\chi X(K_1)$ where $\varepsilon_\chi$ runs through idempotents that correspond to characters $\chi\in \hat\Delta$. Since $\varepsilon_1 X(K_1) = X(K_1)^\Delta$ is the maximal $Z$-quotient of  $X(K_1)$ that is fixed by $\Delta$,  we have $X(K_0)\cong X(K_1)^\Delta$ as $Z$-modules. Thus,  $\mu(X(K_0))=\mu(X(K_1)^\Delta)=0$.  As $K_0$ is totally real, \cite[Theorem 13.31]{Wa} implies that $X(K_0)$ has no free $\Lambda$-quotients. Since moreover $\mu(X(K_0))=0$,  $X(K_1)/\ell X(K_1)$ has no free $\ff{\ell}[[Z]]$-quotients,  contradicting  the solvability of  $(\phi_u,\pi_1)$.
\end{proof}

As a consequence it follows from Proposition \ref{bounded.prop} that in the case $K=\mQ$, $\oline{F}$ is not $Z$-isomorphic to a product of the $Z$-modules $\ff{\ell}[[Z]]$ and $V_n,n\in\mathbb{N}$. Indeed, otherwise the dual $\hat F$ would be a direct sum of $Z$-modules isomorphic to $V_n$ and $\hat V$, but each such direct sum has trivial infinite Ulm invariants.



\subsection{Ulm invariants and finite summands}\label{modules-proofs.sec}


 Proposition \ref{bounded.prop} follows directly from the following lemma which asserts its dual.
 The key to its proof is the following criterion for an $\ff{\ell}[[Z]]$-submodule $E\leq D$ to be a direct summand.  The submodule $E$ is called {\bf pure} if $I^kE=I^kD\cap E$ for all $k\in \mathbb{N}$. By
 \cite[Theorem~7]{Kap}\footnote{Theorem 7 in \cite{Kap} asserts the corresponding statement for $\mZ$-modules. As noted in \cite[\S 12]{Kap} the same proof works for modules over a PID, and in particular over $\ff{\ell}[[Z]]$.}, every pure submodule $E\leq D$ such that $I^NE=0$ for some $N\in\mathbb{N}$ is a direct $\ff{\ell}[[Z]]$-summand of $D$.

  We write $\Ht_E$ to specify that the height is taken within $E$. We shall write $V_n^{\oplus\kappa}$ to denote the direct sum of $\kappa$ copies of $V_n$. 
\begin{lemma}\label{bounded.lem}
Let  $N\in\mathbb{N}\cup\{0\}$ and let $D$ be a  discrete torsion $\ff{\ell}[[Z]]$-module. Then $D=P_N\oplus Q_N$ where
\begin{equation}\label{equ:decomposition}P_N \cong  \bigoplus_{1\leq n\leq N} V_n^{\oplus U_{n-1}(D)},\end{equation}
and $Q_N$ has no direct $\ff{\ell}[[Z]]$-summands of dimension $1\leq d\leq N$.
\end{lemma}
\begin{proof}
We argue by induction on $N$ with $N=0$ being trivial. By induction $D=P_N\oplus Q_N$, where $P_N$ is as in \eqref{equ:decomposition}, and $Q_N$ has no summands of dimension $\leq N$. The induction hypothesis  applied to $Q_N$ also shows that $U_{n-1}(Q_N)=0$ for all $n\leq N$, as otherwise $Q_N$ would have direct $Z$-summands of dimension $1\leq d\leq N$. Hence by \eqref{finite-Ulm.equ}, there are no element in $Q_N^Z$ of height $\leq N-1$ in $Q_N$.


We shall construct $T\leq Q_N$ such that $Q_N=T\oplus Q_{N+1}$, $T=V_{N+1}^{\oplus U_{N}(Q_N)}$ and $Q_{N+1}$ has no $Z$-summands isomorphic to $V_{N+1}$. As $Z$ acts trivially on $Q_N^Z$, we regard $Q_N^Z$ as an $\ff{\ell}$-vector space. Let $V$ be the $\ff{\ell}$-subspace of $Q_N^Z$ consisting of elements of height $>N$, and $U$ a complement of it in $Q_N^Z$. In particular, $U$ is a maximal $\ff{\ell}$-subspace of $Q_{N}^Z$ whose nontrivial elements are of height $N$ in $Q_N$.  Thus $\dim_{\ff{\ell}}U = U_N(Q_N)$. Let $\{u_j\}_{j\in J}$ be an $\ff{\ell}$-basis of $U$; hence $|J|=U_N(Q_N)$.

Let $R:=\ff{\ell}[[Z]]$, $\sigma$ a generator of $Z$, and $x:=\sigma-1$ a generator of the augmentation ideal $I\lhd R$. Since each $u_j$  is of height $N$, we may pick an element $p_j\in  Q_{N}$ such that $x^{N}p_j=u_j$, $j\in J$. Let $T$  be the $R$-submodule $\sum_{j\in J}Rp_j$.
Since $x^Np_j=u_j\neq 0$, and $x^{N+1}p_j=0$,  $Rp_j$ is cyclic of dimension $N+1$ and hence $Rp_j\cong V_{N+1}$, for $j\in J$.

We claim that $T=\oplus_{j\in J} Rp_j\cong \oplus_{j\in J}V_{N+1}$. 
Assume there is a nontrivial linear combination $\sum_{i\leq N, j\in J}a_{i,j}x^ip_j=0$, with $a_{i,j}\in \ff{\ell}$, $i\leq N, j\in J$. Multiplying by $x^{N-i_0}$ where $i_0$ is the minimal number for which $a_{i_0,j}\neq 0$ for some $j$, we obtain a nontrivial linear combination $\sum_{j\in J}b_jx^Np_j=\sum_{j\in J}b_ju_j=0$.
This contradicts the linear independence of $u_j,j\in J$, proving the claim.

We next show that $T$ is a direct $R$-summand of $Q_{N}$. Since all nontrivial elements of $U$ are of height $N$ in $Q_N$, the height of each element in $T$ is the same as its height in $Q_N$. Hence, $T$ is a pure submodule of $Q_N$. Since $I^{N+1}T=0$, \cite[Theorem~7]{Kap} implies that $Q_N= T\oplus  Q_{N+1}$ for some $R$-submodule $Q_{N+1}\leq  Q_{N}$.

Finally, we show that $Q_{N+1}$ has no $R$-summands isomorphic to $V_n$, for $n\leq~N+~1$. Since $Q_{N+1}\leq Q_N$, $\Ht_{Q_{N+1}}(q)\geq N$ for every $q\in  Q_{N+1}^Z$. We claim that $\Ht_{ Q_{N+1}}(q)>~N$ for every $q\in  Q_{N+1}^Z$. Indeed, if $\Ht_{Q_{N+1}}(q)=N$ then for any $u\in U$, one has
$\Ht_{ Q_{N}}(q+u)= \min(\Ht_{ Q_{N}}(q),\Ht_{ Q_{N}}(u))=N,$ contradicting the maximality of $U$ and  proving the claim. As $Q_{N+1}^Z$ has no elements of height $\leq N$,  $Q_{N+1}$ has no $R$-summands isomorphic to $V_n$, for $n\leq N+1$. Setting $P_{N+1}:= P_N\oplus T$, we obtain the desired decomposition $D=P_{N+1}\oplus Q_{N+1}$.
\end{proof}


\bibliographystyle{plain}

\end{document}